\documentclass[final,10pt]{amsart}
\usepackage{amsmath,amsfonts,amssymb,amscd,verbatim}
\usepackage{mathtools}
\usepackage[margin=0.4in]{geometry}
\usepackage{fullpage}
\usepackage{enumerate}

\usepackage{bm}
\usepackage{xcolor}
\usepackage{graphicx}
\usepackage{enumitem}
\usepackage{tikz-cd}
\usepackage[normalem]{ulem}
\usepackage{mathrsfs}

\allowdisplaybreaks[1]

\numberwithin{equation}{section}
\theoremstyle{plain}
\newtheorem{theorem}[equation]{Theorem} 

\newtheorem{corollary}[equation]{Corollary} 
\newtheorem{lemma}[equation]{Lemma}
\newtheorem{proposition}[equation]{Proposition}

\theoremstyle{definition}
\newtheorem{definition}[equation]{Definition}
 
\newtheorem{example}[equation]{Example}
\newtheorem{hypothesis}[equation]{Hypothesis}

\newtheorem{remark}[equation]{Remark}

\DeclareMathOperator\Aut{Aut}

\DeclareMathOperator\chr{char}

\DeclareMathOperator\Ext{Ext}
\DeclareMathOperator\gldim{gldim}
\DeclareMathOperator\GKdim{GKdim}

\DeclareMathOperator\id{id}
\DeclareMathOperator\im{im}

\DeclareMathOperator\lgld{l.gldim}

\DeclareMathOperator\MaxSpec{MaxSpec}
\DeclareMathOperator\orb{orb}
\DeclareMathOperator\ord{ord}

\DeclareMathOperator\reg{reg}

\DeclareMathOperator\Res{Res}
\DeclareMathOperator\rgld{r.gldim}

\DeclareMathOperator\Span{Span}

\DeclareMathOperator\Supp{Supp}

\DeclareMathOperator\wmod{\text{-}wmod}
\DeclareMathOperator\wmodO{\text{-}wmod_{\mathcal{O}}}

\newcommand\CC{\mathbb C}
\newcommand\NN{\mathbb N}
\newcommand\RR{\mathbb R}
\newcommand\ZZ{\mathbb Z}
\newcommand\QQ{\mathbb Q}

\newcommand\cA{\mathcal A}

\newcommand\cC{\mathcal C}
\newcommand\cD{\mathcal D}

\newcommand\cJ{\mathcal J}

\newcommand\cO{\mathcal O}

\newcommand\fg{\mathfrak g}

\newcommand\bd{\mathbf d}
\newcommand\be{\mathbf e}

\newcommand\bbm{\mathbf m}
\newcommand\bbn{\mathbf n}

\newcommand\bq{\mathbf q}

\newcommand\bzero{\mathbf 0}

\newcommand\kk{\Bbbk}

\newcommand\fwa{A_1(\kk)}

\newcommand\qwa{A_n^{\bq,\Lambda}(\kk)}

\newcommand\cnt{\mathcal Z}
\renewcommand{\int}{\mathrm{int}}
\newcommand\inv{^{-1}}
\newcommand\iso{\cong}
\newcommand\niso{\not\cong}

\newcommand\tensor{\otimes}

\newcommand\restrict[1]{\raisebox{-.3ex}{$|$}_{#1}}

\newcommand\grp[1]{\langle #1 \rangle}

\renewcommand{\to}{\ensuremath{\longrightarrow}}

\newcommand{\into}{\hookrightarrow}
\newcommand{\onto}{\twoheadrightarrow}

\title{Fixed rings of twisted generalized Weyl algebras}
\author[Gaddis]{Jason Gaddis}
\address{Miami University, Department of Mathematics, Oxford, Ohio 45056} 
\email{gaddisj@miamioh.edu}

\author[Rosso]{Daniele Rosso}
\address{Department of Mathematics and Actuarial Science, Indiana University Northwest, Gary, IN 46408} 
\email{drosso@iu.edu}

\subjclass[2010]{16W22,16S32,16D70,16P40}
\keywords{Fixed ring, twisted generalized Weyl algebra, tensor product, noetherian ring, weight module}

\begin{document}

\begin{abstract}
Twisted generalized Weyl algebras (TGWAs) are a large family of algebras that includes several algebras of interest for ring theory and representation theory, such as Weyl algebras, primitive quotients of $U(\mathfrak{sl}_2)$, and multiparameter quantized Weyl algebras. In this work, we study invariants of TGWAs under diagonal automorphisms. Under certain conditions, we are able to show that the fixed ring of a TGWA by such an automorphism is again a TGWA. 
We apply our results particularly to $\Bbbk$-finitistic TGWAs of type $(A_1)^n$ and $A_2$, and study properties of their fixed rings, including the noetherian property and simplicity. We also look at the behavior of simple weight modules for TGWAs when restricted to the action of the fixed ring. As an auxiliary result, in order to study invariants of tensor products of TGWAs, we prove that the class of regular, $\mu$-consistent TGWAs is closed under tensor products.
\end{abstract}

\maketitle

\section{Introduction}

A common theme in invariant theory, inspired by the famed Shephard-Todd-Chevalley Theorem, is to study when taking invariants preserves certain algebraic structures \cite{ShTo,Chev}. However, some structures are too restrictive to encapsulate rings of invariants, or fixed rings. For example, Smith has shown that the fixed ring of the first Weyl algebra $A_1(\CC)$ by a finite group is never isomorphic the first Weyl algebra \cite{Sm1}. In fact, by a theorem of Tikaradze, the Weyl algebra is not the fixed ring of any domain \cite{TIK}. Alev and Polo proved a version of Smith's theorem for the $n$th Weyl algebra.

However, by a theorem of Jordan and Wells \cite{JW}, the fixed ring of $A_1(\CC)$ under a diagonal automorphism is always a generalized Weyl algebra (GWA). Won and the first-named author showed that this holds true also for any \emph{linear} automorphism on $A_1(\CC)$ \cite{GW1}. Fixed rings of GWAs have also been studied in \cite{GHo,KK}. It is reasonable then to consider similar questions for higher rank GWAs and their generalizations.

\begin{example}\label{ex.mqwa}
Let $\kk$ be a field. Let $\bq = (q_1,\hdots,q_n) \in (\kk^\times)^n$ and let $\Lambda = (\lambda_{ij}) \in M_n(\kk^\times)$ be a multiplicatively antisymmetric matrix (so $\lambda_{ij}\lambda_{ji}=1$ for all $i,j$). The \emph{(multiparameter) quantized Weyl algebra}, denoted $\qwa$, is the unital $\kk$-algebra generated by $x_1,y_1,\hdots,x_n,y_n$ subject to the following relations for all $i,j \in \{1,\hdots,n\}$ with $i<j$,
\begin{align*}
&y_iy_j=\lambda_{ij} y_jy_i, \qquad x_ix_j=q_i\lambda_{ij}x_jx_i, \\
&x_iy_j=\lambda_{ji}y_jx_i, \qquad x_jy_i=q_i\lambda_{ij}y_ix_j, \\
&x_iy_i-q_iy_ix_i = 1 + \sum_{k=1}^{i-1}(q_k-1)y_kx_k.
\end{align*}
Set $A=\qwa$ and fix $\alpha \in (\kk^\times)^n$ with $m_i = \ord(\alpha_i) < \infty$ for all $i$. Define an automorphism $\phi:\qwa \to \qwa$ by $\phi(x_i)=\alpha_i x_i$ and $\phi(y_i)=\alpha_i\inv y_i$ for all $i$. While $A$ is generated as an algebra over $\kk$ by the $2n$ elements $\{x_i,y_i\}$, the fixed ring $A^\phi$ is generated by the $2n$ elements $\{x_i^{m_i},y_i^{m_i}\}$ as well as elements of the form $x_iy_i$. Hence, it is not clear whether $A^\phi$ is another algebra of the same type.
\end{example}

In light of the previous example, it is reasonable then to consider a larger family of algebras that may encompass such invariant rings. We turn to twisted generalized Weyl algebras (TGWAs), which were originally defined by Mazorchuk and Turowska \cite{MT}. Examples include generalized Weyl algebras (of arbitrary degree), the quantized Weyl algebras of Example \ref{ex.mqwa}, and primitive quotients of $U(\fg)$ for $\fg$ a finite-dimensional semisimple Lie algebra \cite{HVS}. In Section \ref{sec.tgwa} we provide background on TGWAs with some emphasis on those of $\kk$-finitistic type associated to a Cartan matrix $C$. We also prove in Theorem \ref{thm.tensor} that the class of regular TGWAs is closed under taking tensor products.

In Section \ref{sec.fix} we study invariants of TGWAs under diagonal automorphisms defined in Hypothesis \ref{hyp.auto}. Our main result, Theorem \ref{thm.fix}, shows that a certain subring of invariants under such an automorphism is again a TGWA and that regularity and a certain consistency condition pass to this subring.

In Section \ref{sec.A1} we specialize these results to $\kk$-finitistic TGWAs of Cartan type $(A_1)^n$. We show in Theorem \ref{thm.fix1} that this family of TGWAs is stable under taking fixed rings. This includes, in particular, GWAs of arbitrary rank. We also consider ring-theoretic properties that are stable under taking fixed rings. We show in Theorem \ref{thm.ore} that a regular, $\mu$-consistent $\kk$-finitistic TGWA of type $(A_1)^n$ is (left/right) noetherian, and this property passes to its invariants.

Section \ref{sec.A2} considers those $\kk$-finitistic TGWAs of Cartan type $A_2$. Fixed rings in the case of type $A_2$ are not again of type $A_2$, but they are rank 2 TGWAs (Theorem \ref{thm.A2}). We prove a result on the noetherian properties for certain type $A_2$ TGWAs (Theorem \ref{thm.ore2}) and apply this to determine a family of noetherian TGWAs of rank 2 obtained through taking invariants.

Finally, in Section \ref{sec.weight}, we consider the restriction functor on weight modules from a TGWA to its subring of invariants. We explain the general setup and give more specific descriptions for the cases of rank 1 GWAs with infinite orbits (Theorem \ref{thm.rank1}) and for the case of a $\kk$-finitistic TGWA of type $A_2$ (Proposition \ref{prop.rank2mod}).

\subsection*{Acknowledgments}
We thank Jonas T. Hartwig for useful conversations, especially for some of the ideas in the proof of Theorem \ref{thm.tensor}. We also thank the anonymous referee for useful suggestions that improved the final version of the paper.

\subsection*{Notation} 
Throughout, we let $\kk$ be a field.
Unless otherwise noted, all tensor products and algebras are over $\kk$. We denote the set $\{1,\hdots,n\} \subset \NN$ by $[n]$. We denote the center of a ring $R$ by $\cnt(R)$ and the set of regular elements by $R_{\reg}$. We denote the order of an automorphism $\phi$ by $\ord(\phi)$. This notation is also used to denote the order of an invertible element in the group of units.

We denote by $\be_i \in \ZZ^n$ the multi-index with $1$ in the $i$th spot and $0$ elsewhere. Similarly, $\bzero \in \ZZ^n$ denotes the multi-index of all $0$s.

\section{Twisted Generalized Weyl algebras}\label{sec.tgwa}

In this section we recall the definition of twisted generalized Weyl algebras, several of their properties, and common examples. In Theorem \ref{thm.tensor} we establish that these algebras are closed under taking tensor products.

There are now several variations of the definition of a TGWA in the literature at different levels of generality. Hence, we caution the reader to observe the conventions used here.

\begin{definition}\label{defn.tgwa}
Let $n$ be a positive integer. A \emph{twisted generalized Weyl datum (TGWD) of rank $n$} is the triple $(R,\sigma,t)$ where $R$ is a unital $\kk$-algebra, $\sigma=(\sigma_1,\hdots,\sigma_n)$ is an $n$-tuple of commuting ring automorphisms of $R$, and $t=(t_1,\hdots,t_n)$ is an $n$-tuple of nonzero central elements of $R$. Given such a TGWD and $\mu=(\mu_{ij}) \in M_n(\kk^\times)$, the associated \emph{twisted generalized Weyl construction (TGWC)}, $\cC_\mu(R,\sigma,t)$, is the $\kk$-algebra generated over $R$ by the $2n$ indeterminates $X_1^{\pm},\hdots,X_n^{\pm}$ subject to the relations
\begin{align}
\label{eq.tgwc1} &X_i^{\pm} r - \sigma_i^{\pm 1}(r) X_i^{\pm} & &\text{for all $r \in R$ and $i \in [n]$,} \\
\label{eq.tgwc2} &X_i^-X_i^+ - t_i, \qquad X_i^+X_i^- - \sigma_i(t_i) &	&\text{for all $i \in [n]$,} \\
\label{eq.tgwc3} &X_i^+ X_j^- - \mu_{ij} X_j^- X_i^+ & &\text{for all $i,j \in [n]$, $i \neq j$.}
\end{align}
There is a natural $\ZZ^n$-grading on $\cC_\mu(R,\sigma,t)$ obtained by setting $\deg(r)=\bzero$ for all $r \in R$ and $\deg(X_i^{\pm})=\pm \be_i$ for all $i \in [n]$. The associated \emph{twisted generalized Weyl algebra (TGWA)}, $A=\cA_\mu(R,\sigma,t)$, is the quotient $\cC_\mu(R,\sigma,t)/\cJ$ where $\cJ$ is the sum of all graded ideals $J = \bigoplus_{g \in \ZZ^n} J_{g}$ such that $J_\bzero = \{0\}$.
\end{definition}
\noindent When $\mu_{ij}=1$ for all $i \neq j$, we simplify the above notation by omitting $\mu$. We refer to $n$ as the \emph{rank} of the TGWA.
The ideal $\cJ$ in Definition \ref{defn.tgwa} can also be defined in the following way \cite{hart4}:
\begin{align}
\label{eq.ideal}
\cJ = \{ c \in \cC_\mu(R,\sigma,t) ~|~ rc=0 \text{ for some } r \in R_{\reg} \cap \cnt(R) \}.
\end{align}

\begin{example}[Bavula {\cite{B1}}]\label{ex.gwa}
Let $D$ be a unital $\kk$-algebra, $\sigma=(\sigma_1,\hdots,\sigma_n)$ a set of commuting automorphisms of $D$, and $t=(t_1,\hdots,t_n)$ a set of non-zero central regular elements of $D$ such that $\sigma_i(t_j)=t_j$ for all $i \neq j$. The \emph{generalized Weyl algebra (GWA)} $D(\sigma,t)$ of rank $n$ is the TGWA $\cA(R,\sigma,t)$. 

Suppose $i \neq j$, then a computation shows that
$(X_i^+ X_j^+ - X_j^+ X_i^+)(t_i)=0$.
Hence, $X_i^+ X_j^+ - X_j^+ X_i^+ \in \cJ$ and similarly $X_i^- X_j^- - X_j^- X_i^- \in \cJ$. One can show that the ideal $\cJ$ is generated by terms of this form for all distinct $i,j \in [n]$.

As a specific example, consider the rank one GWA $A=D(\sigma,t)$ with $D=\kk[h]$, $\sigma(h)=h-1$, and $t=h$. Set $x=x_1$ and $y=y_1$. Then $yx=h$ and $xy=\sigma(h)=h-1$ so the indeterminate $h$ is redundant and we have $yx=xy+1$. That is $A \iso \fwa$, the classical first Weyl algebra. Note that $\fwa \tensor \cdots \tensor \fwa$ ($n$ copies) is isomorphic to the $n$th classical Weyl algebra $A_n(\kk)$, which can also be realized as the rank $n$ GWA with $D=\kk[h_1,\hdots,h_n]$, $\sigma_i(h_j)=h_j-\delta_{ij}$, and $t_i=h_i$. In general, GWAs are closed under tensor products. We have $D(\sigma,t) \tensor D'(\sigma',t') \iso (D \tensor D')(\sigma \tensor \sigma', t \cup t')$.
\end{example}

\begin{example}\label{ex.mqwa2}
The quantized Weyl algebras discussed in Example \ref{ex.mqwa} are examples of TGWAs.
Let $R=\kk[h_1,\hdots,h_n]$. Define $\sigma \in (\Aut R)^n$ and $\mu \in M_n(\kk^\times)$ by
\[
\sigma_i(h_j) = \begin{cases}
	h_j & \text{if $j<i$} \\
	1+q_ih_i + \sum_{k=1}^{i-1}(q_k-1)h_k & \text{if $j=i$} \\
	q_ih_j & \text{if $j>i$},\end{cases} \qquad
\mu_{ij} = \begin{cases}
	\lambda_{ij} & \text{if $i<j$} \\
	q_j\lambda_{ji} & \text{if $i>j$}.\end{cases}
\]
Set $t_i=h_i$. There is an isomorphism $A_\mu(R,\sigma,t) \to \qwa$ determined by $X_i^+ \mapsto x_i$, $X_i^- \mapsto y_i$, and $h_i \mapsto y_ix_i$. 
\end{example}

It may happen that $A_\mu(R,\sigma,t)$ is in fact trivial and this is a case we wish to avoid.

\begin{definition}
We say the TGWD $(R,\sigma,t)$ is \emph{regular} if $t_i \in R_{\reg}$ for all $i \in [n]$. For a parameter matrix $\mu$, we say $(R,\sigma,t)$ is \emph{$\mu$-consistent} if the canonical map $R \to A_\mu(R,\sigma,t)$ is injective. When all $\mu_{ij}=1$ and $(R,\sigma,t)$ satisfies the above then we simply say that $(R,\sigma,t)$ is \emph{consistent}. We say a TGWA $A_\mu(R,\sigma,t)$ is regular (resp. $\mu$-consistent) if the underlying TGWD is regular (resp. $\mu$-consistent).
\end{definition}

If $(R,\sigma,t)$ is $\mu$-consistent for some parameter matrix $\mu$, then $A_\mu(R,\sigma,t)$ is necessarily non-trivial. If $(R,\sigma,t)$ is a regular TGWD, then by \cite[Theorem 6.2]{FH2}, $(R,\sigma,t)$ is $\mu$-consistent if and only if the following equations hold:
\begin{align}
\label{eq.cons1} &\sigma_i\sigma_j(t_it_j)=\mu_{ij}\mu_{ji}\sigma_i(t_i)\sigma_j(t_j) \text{ for all distinct $i,j \in [n]$,} \\
\label{eq.cons2} &t_j\sigma_i\sigma_k(t_j) = \sigma_i(t_j)\sigma_k(t_j) \text{ for all pairwise distinct $i,j,k \in [n]$.}
\end{align}

We now recall some properties of TGWAs. Let $(R,\sigma,t)$ be a TGWD. A \emph{monic monomial} in $C=\cC_\mu(R,\sigma,t)$ (or $A=\cA_\mu(R,\sigma,t)$) is a word in the set $\{X_1^{\pm},\cdots,X_n^{\pm}\}$. A \emph{reduced monomial} in $C$ (or $A$) is a monic monomial of the form
\begin{align}\label{eq.reduced}
X_{i_1}^- \cdots X_{i_k}^- X_{j_1}^+ \cdots X_{j_\ell}^+ \text{ with } 
\{i_1,\hdots,i_k\} \cap \{j_1,\hdots,j_\ell\} = \emptyset.
\end{align}

\begin{theorem}\label{thm.props}
Let $(R,\sigma,t)$ be a regular, $\mu$-consistent TGWD. Let $C=\cC_\mu(R,\sigma,t)$ be the associated TGWC and $A=\cA_\mu(R,\sigma,t)$ the associated TGWA.
\begin{enumerate}
\item \cite[Proposition 2.9]{FH1} The TGWA $A$ is a domain if and only if $R$ is a domain.
\item \cite[Lemma 3.2]{hart1} The TGWC $C$ (or the TGWA $A$) is generated as a left (and as a right) $R$-module by the reduced monomials.
\end{enumerate}
\end{theorem}
\begin{proof}
We remark only briefly on (2). The given reference only states the proof for TGWCs, but it easily extends to TGWAs. In particular, given a word $Z_1 \cdots Z_k$ where each $Z_i = X_i^+$ or $X_i^-$, then one can use the relations to push all $X_i^-$ to the left. If an $X_i^-$ and $X_i^+$ appear consecutively, then they can be reduced to an element in $R$ and moved to the left.
\end{proof}

It follows directly from Theorem \ref{thm.props} (2) that, under the standard $\ZZ^n$-grading, $C_\bd$ is generated as a left (and as a right) $R$-module by the reduced monomials of degree $\bd \in \ZZ^n$. In particular, $C_\bzero=R$. The same statements hold for $A$.

\subsection{$\kk$-finitistic TGWAs}
The examples to which we apply our main theorem fall into a specific but important class of TGWAs which we define now. These first appeared formally in \cite{hart3}.

\begin{definition}\label{def:k-finite}
Let $A=\cA_\mu(R,\sigma,t)$ be a regular, $\mu$-consistent TGWA. For $i,j \in [n]$, define
\[ V_{ij} = \Span_\kk\{\sigma_i^k(t_j) ~|~ k \in \ZZ\}.\]
Then $A$ is \emph{$\kk$-finitistic} if $\dim_\kk V_{ij} < \infty$ for all $i,j$.
\end{definition}

\begin{remark}\label{rem.kfin}
The original definition of $\kk$-finitistic TGWAs does not a priori require regularity and $\mu$-consistency. However, as most key results regarding them do make this assumption, we have built it into our definition.
\end{remark}

Let $A=\cA_\mu(R,\sigma,t)$ be a $\kk$-finitistic TGWA.
For each $i,j$, let $p_{ij} \in \kk[x]$ denote the minimal polynomial for $\sigma_i$ acting on $V_{ij}$. Associated to this data we define the matrix $C_A = (a_{ij})$ by
\[ a_{ij} = \begin{cases}2 & \text{ if $i=j$} \\ 1-\deg p_{ij} & \text{ if $i \neq j$.}\end{cases}\]
By \cite[Theorem 3.2(a)]{FH1}, $C_A$ is a generalized Cartan matrix.

The next example was first mentioned in \cite[Example 1]{MT} and a full presentation was given in \cite[Example 6.3]{hart3}.

\begin{example}\label{ex.A2simple}
Take $n=2$, $R=\kk[h]$, $\sigma=(\sigma_1,\sigma_2)$, where $\sigma_1(h)=h+1$,  $\sigma_2(h)=h-1$, $t=(t_1,t_2)=(h,h+1)$, $\mu_{12}=\mu_{21}=1$. The corresponding $A=\cA(R,\sigma,t)$ is a $\kk$-finitistic TGWA of type $A_2$ (that is, the corresponding Cartan matrix is $\left[\begin{smallmatrix}2 & -1 \\ -1 & 2 \end{smallmatrix}\right]$). By \cite[Theorem 4.4]{hart3}, $A$ is isomorphic to the $\kk$-algebra with generators $h,X_1^{\pm},X_2^{\pm}$ and relations given by \eqref{eq.tgwc1}-\eqref{eq.tgwc3} plus the additional relations
\begin{align*}
(X_1^\pm)^2X_2^\pm-2X_1^\pm X_2^\pm X_1^\pm +X_2^\pm(X_1^\pm)^2 &=0 \\
(X_2^\pm)^2X_1^\pm-2X_2^\pm X_1^\pm X_2^\pm +X_1^\pm(X_2^\pm)^2 &=0.
\end{align*}
\end{example}

\subsection{Tensor products}

As mentioned in Example \ref{ex.gwa}, the tensor products of two GWAs is again a GWA \cite{B1}. We extend this construction to TGWAs and use this to study fixed rings of tensor products of TGWAs.

\begin{theorem}\label{thm.tensor}
Let $A=A_\mu(R,\sigma,t)$ be a regular, $\mu$-consistent TGWA of rank $m$. Let $A'=A_{\mu'}(R',\sigma',t')$ be a regular, $\mu'$-consistent TGWA of rank $n$. Define $\eta=(\eta_{ij}) \in M_{m+n}(\kk^\times)$ by
\[
\eta_{ij} = \begin{cases}
\mu_{ij}& \text{if } 1 \leq i,j \leq m \\
\mu_{(i-m)(j-m)}' & \text{if } m<i,j \leq m+n \\
1 & \text{otherwise}.
\end{cases}\]
Then $A \tensor A'$ is a regular, $\eta$-consistent TGWA of rank $m+n$
\end{theorem}
\begin{proof}
We construct a new TGWC and show that $A \tensor A'$ is (isomorphic to) its corresponding TGWA. 

Let $S=R \tensor R'$. Let $\rho=(\rho_1,\hdots,\rho_{m+n})$ and $w=(w_1,\hdots,w_{m+n})$ where
\[ 
\rho_i = \begin{cases}
\sigma_i \tensor \id_{R'} & \text{if } i\leq m \\
\id_R \tensor \sigma_{i-m}'& \text{if } i>m,
\end{cases}\qquad
w_i = \begin{cases}
t_i \tensor 1 & \text{if } i\leq m \\
1 \tensor t_{i-m}'& \text{if } i>m.
\end{cases}\]
It is clear that the $\rho_i$ are commuting automorphisms of $S$ and that the $w_i$ are nonzero and central in $S$. Thus, $(S,\rho,w)$ is a TGWD. 

We claim that $(S,\rho,w)$ is regular. If $s\in S= R\otimes R'$, then we can write $s=\sum_j r_j\otimes r'_j$, with $\{r'_j\}$ $\kk$-linearly independent in $R'$. For $1\leq i\leq m$, suppose that $w_is=0$, then
\[0=w_is=(t_i\otimes 1)\cdot \left(\sum_j r_j\otimes r'_j\right)=\sum_j t_ir_j\otimes r'_j\]
since $\{r'_j\}$ are $\kk$-linearly independent, this implies that $t_ir_j=0$ for all $j$, but since $t_i$ is regular in $R$, this implies that $r_j=0$ for all $j$, hence $s=\sum_j r_j\otimes r'_j=0$, which proves the claim. An entirely analogous argument proves that $w_i$, with $m+1\leq i\leq m+n$, is also regular in $S=R\otimes R'$.

Since $(R,\sigma,t)$ is $\mu$-consistent, then clearly $(S,\rho,w)$ satisfies \eqref{eq.cons1} with respect to $\eta$ for $i,j \leq m$. Similarly, since $(R',\sigma',t')$ is $\mu'$-consistent, then $(S,\rho,w)$ satisfies \eqref{eq.cons1} with respect to $\eta$ for $i,j > m$. If $i \leq m$ and $j>m$ (the case $i > m$ and $j \leq m$ being analogous), then 
\[\rho_i\rho_j(w_iw_j)=1\cdot 1\cdot(\sigma_i(t_i)\otimes 1)(1\otimes \sigma'_j(t'_j))=\eta_{ij}\eta_{ji}\rho_i(w_i)\rho_j(w_j).\]
Thus, $(S,\rho,w)$ satisfies \eqref{eq.cons1} with respect to $\eta$. A similar argument, but with more cases to check, shows that $(S,\rho,w)$ satisfies \eqref{eq.cons2} with respect to $\eta$.
Hence, since it is regular, $(S,\rho,w)$ is $\eta$-consistent.

Let $\widehat{C}=C_\eta(S,\rho,w)$ be the TGWC corresponding to $(S,\rho,w)$ and $\eta$. Denote the canonical ideal associated to $\widehat{C}$ by $\widehat{\cJ}$. Denote those associated to $C=C_\mu(R,\sigma,t)$ and $C'=C_{\mu'}(R',\sigma',t')$ by $\cJ$ and $\cJ'$, respectively. 

Since $(S,\rho,w)$ is $\eta$-consistent, the degree zero component of $\widehat{C}$ is $S=R\otimes R'$. We can then define $\psi:\widehat{C} \to A \tensor A'$ to be the map that is the identity on $R \tensor R'$ and maps 
\[
X_i^{\pm} \mapsto \begin{cases}
X_i^{\pm} \tensor 1 & \text{if } i\leq m \\
1 \tensor X_{i-m}^{\pm} & \text{if } i>m.
\end{cases}\]
It is easy to verify that $\psi$ is a surjective homomorphism. Since $\psi$ is a $\ZZ^{m+n}$-graded map, then $\ker\psi$ is a graded ideal. Since $\psi$ restricts to the identity on $R \tensor R'$, then $(\ker\psi)_\bzero = \{0\}$. That is, $\ker\psi \subset \widehat{J}$.

For the other direction, let $c \in \widehat{J}$ be homogeneous of degree $g=(g_1,g_2)\in\ZZ^m\times\ZZ^n=\ZZ^{m+n}$ and write $\psi(c) = \sum_{i=1}^n a_i \tensor b_i$ such that the $\{a_i\}$ are $\kk$-linearly independent (as are the $\{b_i\}$). Because $\psi$ respects the grading on $\widehat{C}$ and $A \tensor A'$, then we may assume that the $a_i$'s all have degree $g_1$ while the $b_i$'s all have degree $g_2$. Choose a monic monomial $a^* \in A$ such that $\deg(a_1^*)=-g_1$ and a monic monomial $b^* \in A'$ such that $\deg(b_1^*)=-g_2$. Thus, $(a^* \tensor b^*)\psi(c) \in R \tensor R'$. But taking a preimage $c^*$ of $a^* \tensor b^*$ gives $c^* c \in (R \tensor R')\cap\hat{J}=\{0\}$ and hence $(a^* \tensor b^*)\psi(c)=0$. If $\psi(c)\neq 0$, then there is a $\kk$-relation amongst the $\{a^*a_i\}$ or the $\{b^*b_i\}$. Suppose we are in the first case and the second case is similar. Then there are $k_i \in \kk$ such that
\[ a^* (k_1 a_1 + \cdots + k_n a_n) = 0.\]
Since $A$ is a regular TGWA, then \cite[Theorem 4.3 (iv)]{HO} implies that $k_1 a_1 + \cdots + k_n a_n = 0$, a contradiction. Thus, $\ker\psi=\widehat{J}$ and the induced map $\widehat{C}/\widehat{J} \to A \tensor A'$ is an isomorphism.
\end{proof}

\begin{corollary}\label{cor.tensorprops}
We keep the notation introduced in Theorem \ref{thm.tensor}. Then $A \tensor A'$ is $\kk$-finitistic if and only if $A$ and $A'$ are. In this case, the Cartan matrix for the $\kk$-finitistic TGWA $A\otimes A'$ is the block diagonal matrix
\[C_{A\otimes A'}=\begin{bmatrix}C_A & 0 \\ 0 & C_{A'} \end{bmatrix}.\]
\end{corollary}
\begin{proof}
Suppose $A$ and $A'$ are $\kk$-finitistic. If $i \leq m$ and $j > m$, then $\rho_i(w_j)=w_j$, so $\dim_{\kk} V_{i,j} = 1$. Thus, $\dim_{\kk} V_{i,j} < \infty$ for all $i,j$. The converse is similar. 
\end{proof}

\section{Automorphisms of TGWAs and fixed rings}\label{sec.fix}

In this section, we study the fixed ring structure of TGWAs under certain automorphisms defined below.

\begin{hypothesis}\label{hyp.auto}
Let $C=\cC_\mu(R,\sigma,t)$ be a $\mu$-consistent TGWC of rank $n$. Assume $\phi \in \Aut(C)$ satisfies
\begin{enumerate}
\item $\phi\restrict{R}$ is an automorphism of $R$ with $\ell=\ord(\phi\restrict{R})<\infty$;
\item for each $i$, $\phi(X_i^{\pm}) = \alpha_i^{\pm 1} X_i^{\pm}$ for some $\alpha_i \in \kk^\times$ with $m_i=\ord(\alpha_i)<\infty$;
\item the integers $\ell,m_1,\hdots,m_n$ are pairwise relatively prime;
\item either $R$ is a commutative domain or  $\phi\restrict{R}=\id_R$.
\end{enumerate}
It is clear that $\phi$, when extended to all of $C$ in the natural way, is indeed an automorphism of $C$.
\end{hypothesis}

In Hypothesis \ref{hyp.auto}, the assumption that $C$ is $\mu$-consistent is necessary for condition (1) so that $R \hookrightarrow C$. Condition (2) implies that
\[t_i = X_i^-X_i^+ = (\alpha_i^- X_i^-)(\alpha_i^+ X_i^+) = \phi(X_i^-)\phi(X_i^+) = \phi(X_i^-X_i^+) = \phi(t_i).\]
Condition (4) is not necessary for most results of this section, but it is for the proof of the main result (Theorem \ref{thm.fix}).

\begin{lemma}\label{lem.fix-ideal}
Let $C=\cC_\mu(R,\sigma,t)$ be a TGWC of rank $n$ and suppose $\phi \in \Aut(C)$ satisfies Hypothesis \ref{hyp.auto}. Then $\phi(\cJ) = \cJ$. Hence, $\phi$ descends to an automorphism of $\cA_\mu(R,\sigma,t)$ (which we will denote again by $\phi$).
\end{lemma}
\begin{proof}
We use the description of $\cJ$ in \eqref{eq.ideal}. Let $c \in \cJ$. There exists $r \in R_{\reg} \cap \cnt(R)$ such that $rc=0$. We have $0=\phi(rc)=\phi(r)\phi(c)$. Since $\phi$ preserves $R_{\reg}$ and $\cnt(R)$, then it follows that $\phi(c) \in \cJ$. Thus, $\phi(\cJ) \subset \cJ$. Applying $\phi\inv$ gives the result.
\end{proof}

In light of Lemma \ref{lem.fix-ideal}, we will also use Hypothesis \ref{hyp.auto} with regards to automorphisms of $\cA_\mu(R,\sigma,t)$.

\begin{remark}\label{rem.phisig}
Suppose $(R,\sigma,t)$ is a $\mu$-consistent TGWD and $\phi$ is an automorphism of $C_\mu(R,\sigma,t)$ (or $A_\mu(R,\sigma,t)$) as in Hypothesis \ref{hyp.auto}. Then for any $r \in R$,
\[ \alpha_i \phi\sigma_i(r) X_i^+ =\phi(\sigma_i(r) X_i^+) = \phi(X_i^+ r) = \phi(X_i^+) \phi(r) = \alpha_i X_i^+ \phi(r) = \alpha_i \sigma_i\phi(r) X_i^+.\]
Hence, $\phi\sigma_i(r)=\sigma_i\phi(r)$ for all $r \in R$ and for all $i\in[n]$. It follows that $\sigma_i|_{R^\phi}\in \Aut(R^\phi)$ for all $i=1,\ldots,n$.
\end{remark}

Let $(R,\sigma,t)$ be a TGWD, let $A=A_\mu(R,\sigma,t)$ be the associated TGWA, and let $\phi \in \Aut(A)$ be as in Hypothesis \ref{hyp.auto}. Because we have assumed $\mu$-consistency, then it is clear that $R^\phi \subset A^\phi$ and $(X_i^{\pm})^{m_i} \in A^\phi$ for each $i$. It will turn out in certain cases that these elements generate $A^\phi$, but we do not make that assumption yet.

For $i \in [n]$,
\[
(X_i^-)^{m_i}(X_i^+)^{m_i}
	= (X_i^-)^{m_i-1}(t_i)(X_i^+)^{m_i-1}
	= \sigma_i^{-(m_i-1)}(t_i) (X_i^-)^{m_i-1}(X_i^+)^{m_i-1}.
\]
So, by induction, $(X_i^-)^{m_i}(X_i^+)^{m_i} = s_i$ where
\begin{align}\label{eq.si}
s_i = \prod_{k=0}^{m_i-1} \sigma_i^{-k}(t_i).
\end{align}

Let $s=(s_1,\hdots,s_n)$, $\tau =(\tau_1,\ldots,\tau_n)= (\sigma_1^{m_1},\hdots,\sigma_n^{m_n})$, and $\nu=(\nu_{ij})=(\mu_{ij}^{m_im_j})$ for $i \neq j$.

\begin{lemma}\label{lem.nucons}
Let $C=\cC_\mu(R,\sigma,t)$ be a $\mu$-consistent TGWC of rank $n$ and suppose $\phi \in \Aut(C)$ satisfies Hypothesis \ref{hyp.auto}. Then $(R^\phi,\tau,s)$ is a TGWD. Moreover, we have the following:
\begin{enumerate}
    \item If $(R,\sigma,t)$ is regular, then so is $(R^\phi,\tau,s)$.
    \item If $(R,\sigma,t)$ satisfies \eqref{eq.cons1} with respect to $\mu$, then $(R^\phi,\tau,s)$ satisfies \eqref{eq.cons1} with respect to $\nu$.
\end{enumerate}
\end{lemma}
\begin{proof}
Since the $\sigma_i$ commute, then certainly so do their powers. Hence, $\tau$ is an $n$-tuple of commuting automorphisms of $R^\phi$. The $t_i$ are central in $R$ and since $\sigma$ preserves the center of $R$, then $\sigma\inv(t_i)$ is central as well, it follows that $s_i$ is central in $R$. Now since $(X_i^-)^{m_i}$ and $(X_i^+)^{m_i}$ are clearly fixed by $\phi$ and $(X_i^-)^{m_i}(X_i^+)^{m_i}=s_i$, then $s_i \in R^{\phi}$. Note that a similar argument shows that $s_i$ is regular in $R^\phi$ if $t_i$ is regular in $R$.

Now suppose $(R,\sigma,t)$ satisfies \eqref{eq.cons1} with respect to $\mu$. It remains to show that $(R^\phi,\tau,s)$ satisfies \eqref{eq.cons1} with respect to $\nu$. We need to show that
\begin{equation}\label{eq.proofcons} \tau_i\tau_j(s_is_j) = \nu_{ij}\nu_{ji}\tau_i(s_i)\tau_j(s_j)\end{equation}
that is
\[
\sigma_i^{m_i}\sigma_j^{m_j}\left(\prod_{v=0}^{m_i-1}\sigma_i^{-k}(t_i)\prod_{u=0}^{m_j-1}\sigma_j^{-u}(t_j)\right)
= \mu_{ij}^{m_im_j}\mu_{ji}^{m_im_j}\sigma_i^{m_i}\left(\prod_{v=0}^{m_i-1}\sigma_i^{-v}(t_i)\right)\sigma_j^{m_j}\left(\prod_{u=0}^{m_j-1}\sigma_j^{-u}(t_j)\right).
\]
We prove this by induction on $m_i$ and $m_j$. When $m_i=m_j=1$ this is exactly \eqref{eq.cons1}. For simplicity, we let $q=\mu_{ij}\mu_{ji}$ in our computations. Now suppose that $m_i=1$ and $m_j>1$ (notice that due to the symmetry of the equation the same applies if we switch the two indices). We have
\begin{align*}
\sigma_i\sigma_j^{m_j}\left(t_i\prod_{u=0}^{m_j-1}\sigma_j^{-u}(t_j)\right)
&= \sigma_j^{m_j}\sigma_i(t_i)\sigma_j^{m_j}\sigma_i(t_j)\sigma_i\sigma_j^{m_j}\left(\prod_{u=1}^{m_j-1}\sigma_j^{-u}(t_j)\right) \\
&=\sigma_j^{m_j-1}(\sigma_i\sigma_j(t_it_j))\sigma_i\sigma_j^{m_j-1}\left(\prod_{u=0}^{m_j-2}\sigma_j^{-u}(t_j)\right)\\
&\stackrel{\mathclap{\eqref{eq.cons1}}}{=}
\sigma_j^{m_j-1}(q\sigma_i(t_i)\sigma_j(t_j))\sigma_i\sigma_j^{m_j-1}\left(\prod_{u=0}^{m_j-2}\sigma_j^{-u}(t_j)\right) \\
&=q \sigma_j^{m_j}(t_j)\sigma_i\sigma_j^{m_j-1}(t_i)\sigma_i\sigma_j^{m_j-1}\left(\prod_{u=0}^{m_j-2}\sigma_j^{-u}(t_j)\right) \\
&=q \sigma_j^{m_j}(t_j)\sigma_i\sigma_j^{m_j-1}\left(t_i\prod_{u=0}^{m_j-2}\sigma_j^{-u}(t_j)\right)\\
&\stackrel{\mathclap{\text{(ind hyp)}}}{=} \quad
q \sigma_j^{m_j}(t_j)q^{m_j-1}\sigma_i(t_i)\sigma_j^{m_j-1}\left(\prod_{u=0}^{m_j-2}\sigma_j^{-u}(t_j)\right) \\
&=q^{m_j}\sigma_i(t_i)\sigma_j^{m_j}(t_j)\sigma_j^{m_j}\left(\prod_{u=1}^{m_j-1}\sigma_j^{-u}(t_j)\right)\\
&=q^{m_j}\sigma_i(t_i)\sigma_j^{m_j}\left(\prod_{u=0}^{m_j-1}\sigma_j^{-u}(t_j)\right)
\end{align*}
which shows that \eqref{eq.proofcons} is true in our case. Now suppose that $m_i, m_j>1$, then
\begin{align*}
& \sigma_i^{m_i}\sigma_j^{m_j}\left(\prod_{v=0}^{m_i-1}\sigma_i^{-v}(t_i)\prod_{u=0}^{m_j-1}\sigma_j^{-u}(t_j)\right)\\
&=\sigma_i^{m_i}\sigma_j^{m_j}(t_it_j)\sigma_i^{m_i}\sigma_j^{m_j}\left(\prod_{v=1}^{m_i-1}\sigma_i^{-v}(t_i)\prod_{u=1}^{m_j-1}\sigma_j^{-u}(t_j)\right) \\
&=\sigma_i^{m_i-1}\sigma_j^{m_j-1}(\sigma_i\sigma_j(t_it_j))\sigma_i^{m_i}\sigma_j^{m_j}\left(\prod_{v=1}^{m_i-1}\sigma_i^{-v}(t_i)\right)\sigma_i^{m_i}\sigma_j^{m_j}\left(\prod_{u=1}^{m_j-1}\sigma_j^{-u}(t_j)\right) \\
&\stackrel{\mathclap{\eqref{eq.cons1}}}{=}
\sigma_i^{m_i-1}\sigma_j^{m_j-1}(q\sigma_i(t_i)\sigma_j(t_j))\sigma_i^{m_i-1}\sigma_j^{m_j}\left(\prod_{v=0}^{m_i-2}\sigma_i^{-v}(t_i)\right)\sigma_i^{m_i}\sigma_j^{m_j-1}\left(\prod_{u=0}^{m_j-2}\sigma_j^{-u}(t_j)\right) \\
&= q \sigma_i^{m_i}\sigma_j^{m_j-1}(t_i)\sigma_i^{m_i-1}\sigma_j^{m_j}(t_j)\sigma_i^{m_i-1}\sigma_j^{m_j}\left(\prod_{v=0}^{m_i-2}\sigma_i^{-v}(t_i)\right)\sigma_i^{m_i}\sigma_j^{m_j-1}\left(\prod_{u=0}^{m_j-2}\sigma_j^{-u}(t_j)\right) \\
&= q \sigma_i^{m_i-1}\sigma_j^{m_j}\left(t_j\prod_{v=0}^{m_i-2}\sigma_i^{-v}(t_i)\right)\sigma_i^{m_i}\sigma_j^{m_j-1}\left(t_i\prod_{u=0}^{m_j-2}\sigma_j^{-u}(t_j)\right) \\
&= q \sigma_j^{m_j-1}\sigma_i^{m_i-1}\sigma_j\left(t_j\prod_{v=0}^{m_i-2}\sigma_i^{-v}(t_i)\right)\sigma_i^{m_i-1}\sigma_i\sigma_j^{m_j-1}\left(t_i\prod_{u=0}^{m_j-2}\sigma_j^{-u}(t_j)\right) \\
&\stackrel{\mathclap{\text{(ind hyp)}}}{=} \quad
q \sigma_j^{m_j-1}\left(q^{m_i-1}\sigma_i^{m_i-1}\left(\prod_{v=0}^{m_i-2}\sigma_i^{-v}(t_i)\right)\sigma_j(t_j)\right)\sigma_i^{m_i-1}\left(q^{m_j-1}\sigma_i(t_i)\sigma_j^{m_j-1}\left(\prod_{u=0}^{m_j-2}\sigma_j^{-u}(t_j)\right)\right)\\
&= q^{m_i+m_j-1}\sigma_j^{m_j}(t_j)\sigma_i^{m_i}(t_i)\sigma_i^{m_i-1}\sigma_j^{m_j-1}\left(\prod_{v=0}^{m_i-2}\sigma_i^{-v}(t_i)\prod_{u=0}^{m_j-2}\sigma_j^{-u}(t_j)\right)\\
&\stackrel{\mathclap{\text{(ind hyp)}}}{=}
\quad q^{m_i+m_j-1}\sigma_j^{m_j}(t_j)\sigma_i^{m_i}(t_i)q^{(m_i-1)(m_j-1)}\sigma_i^{m_i-1}\left(\prod_{v=0}^{m_i-2}\sigma_i^{-v}(t_i)\right)\sigma_j^{m_j-1}\left(\prod_{u=0}^{m_j-2}\sigma_j^{-u}(t_j)\right)\\
&= q^{m_im_j}\sigma_i^{m_i}(t_i)\sigma_i^{m_i}\left(\prod_{v=1}^{m_i-1}\sigma_i^{-v}(t_i)\right)\sigma_j^{m_j}(t_j)\sigma_j^{m_j}\left(\prod_{u=1}^{m_j-1}\sigma_j^{-u}(t_j)\right)\\
&= q^{m_im_j}\sigma_i^{m_i}\left(\prod_{v=0}^{m_i-1}\sigma_i^{-v}(t_i)\right)\sigma_j^{m_j}\left(\prod_{u=0}^{m_j-1}\sigma_j^{-u}(t_j)\right)
\end{align*}
which is what we wanted to show.
\end{proof}

\begin{remark}
In general we do not know if \eqref{eq.cons2} passes down to fixed rings. However, in cases we consider, \eqref{eq.cons1} is sufficient for $\nu$-consistency of the fixed ring. In fact, it was proved in \cite{HR1} that when $R$ is a polynomial ring and $\sigma$ consists of shifts in the polynomial variables, \eqref{eq.cons1} implies \eqref{eq.cons2}.
\end{remark}

\begin{lemma}\label{lem.tgwc}
Let $C=\cC_\mu(R,\sigma,t)$ be a $\mu$-consistent TGWC of rank $n$ and suppose $\phi \in \Aut(C)$ satisfies Hypothesis \ref{hyp.auto}. Let $D$ be the subalgebra of $C^\phi$ generated by $R^\phi$ and the $(X_i^{\pm})^{m_i}$. If $(R,\sigma,t)$ is a regular TGWD, then $D \iso \cC_\nu(R^\phi,\tau,s)$ is a regular TGWC satisfying \eqref{eq.cons1}.
\end{lemma}
\begin{proof}
By Lemma \ref{lem.nucons}, we need only verify that these generators satisfy the necessary relations for $D$ to be a TGWC. Recall by \eqref{eq.si}, $(X_i^-)^{m_i}(X_i^+)^{m_i} = s_i$. A similar argument shows that 
\[
(X_i^+)^{m_i}(X_i^-)^{m_i} 
    = \prod_{k=1}^{m_i} \sigma_i^{k}(t_i)
    = \sigma_i^{m_i}(s_i).\]

The remaining relations are easily verified. Let $i,j \in [n]$ with $i \neq j$. Then
\begin{align*}
&(X_i^{\pm})^{m_i} r - \sigma_i^{\pm m_i}(r) (X_i^{\pm})^{m_i} 
& &\text{for all $r \in R^\phi$ and $i \in [n]$,} \\
&(X_i^+)^{m_i}(X_j^-)^{m_j}  = \mu_{ij}^{m_im_j}(X_j^-)^{m_j}(X_i^+)^{m_i} & &\text{for all $i,j \in [n]$, $i \neq j$.}\qedhere
\end{align*}
This shows that we have a surjective map $\cC_\nu(R^\phi,\tau,s)\onto D$, but since the composition 
\[\cC_\nu(R^\phi,\tau,s)\onto D\into C=C_\mu(R,\sigma,t)\] 
is injective, the map is an isomorphism.
\end{proof}

\begin{example}\label{ex.rank2}
Let $\mathrm{char}\kk = 0$, and consider the rank 2 TGWC defined by $R=\kk[h]$, $(\beta_1,\beta_2) \in \kk^2$, $\sigma_i(h)=h-\beta_i$, and $t=(p_1,p_2) \in (R\backslash\{0\})^2$. Here we take all $\mu_{ij}=1$. Notice that, by \cite{HR0}, for this to be consistent we need $\frac{\beta_1}{\beta_2}=-\frac{k_1}{k_2}$ with $k_1$, $k_2$ positive integers, and that if we take $(k_1,k_2)=1$, then we necessarily have $(\deg p_1,\deg p_2)=a\cdot (k_1,k_2)$, in particular $\deg p_i\geq k_i$ for $i=1,2$.
The corresponding TGWA is a \emph{noncommutative (type $A \times A$) Kleinian fiber product} \cite{hart4} and generalizes the noncommutative Kleinian singularities studied by Hodges \cite{H1}. 
Define $\phi \in \Aut(R)$ as in Hypothesis \ref{hyp.auto} with $\phi(r)=r$ for all $r \in R$. Then it follows from the above that $D$ is a TGWC that also gives rise to a noncommutative Kleinian fiber product, with parameters $\tau_i(u)=u- m_i \beta_i$ and $s_i = \prod_{k=0}^{m_i-1} \sigma_i^{-k}(p_i)$. In particular $\deg s_i=m_i\cdot\deg p_i$.
\end{example}

The extra conditions on $R$ in Hypothesis \ref{hyp.auto} (4) are needed in the following proof because the center and the set of regular elements are not well-controlled under taking fixed rings.

\begin{theorem}\label{thm.fix}
Let $A=\cA_\mu(R,\sigma,t)$ be a $\mu$-consistent TGWA of rank $n$, let $\phi\in\Aut(A)$ satisfying Hypothesis \ref{hyp.auto}, and let $B$ be the subalgebra of $A^\phi$ generated by $R^\phi$ and the $(X_i^{\pm})^{m_i}$. Moreover, $B$ is regular if $A$ is regular and satisfies \eqref{eq.cons1} if $A$ does.
\end{theorem}
\begin{proof}
Let $C=\cC_\mu(R,\sigma,t)$ be the TGWC inducing $A$ and let $D$ be as in Lemma \ref{lem.tgwc}. Then $D$ is a TGWC. Let $\cJ'$ be canonical ideal associated to $D$ so that $D/\cJ'$ is a TGWA. We claim $B \iso D/\cJ'$ so that there is a commutative diagram:
\[\begin{tikzcd}        
    D \arrow[hook]{r}{\iota} \arrow[dotted, two heads]{d} & C  \arrow[two heads]{d}{\pi} \\
    B \arrow[hook]{r} & A  
\end{tikzcd}\]
It will then follow that $B$ is a TGWA of rank $n$. The regularity of $B$ along with \eqref{eq.cons1} follow from Lemma \ref{lem.tgwc}.

The composition $\pi\iota$ defines a map $D \to A$, and clearly $\im(\pi\iota) \subset B$. Define $\Psi:D \to B$ in this way and observe that $\Psi$ is surjective. The kernel of this map is exactly $\cJ \cap D$. Thus, it suffices to prove that $\cJ' = \cJ \cap D$.

We use the description of $\cJ$ as in \eqref{eq.ideal}. Let $c \in \cJ \cap D$ so that $rc=0$ for some nonzero $r \in  R_{\reg} \cap \cnt(R)$. Recall that $\ell$ is the order of $\phi$ restricted to $R$. Set $\hat{r}=\prod_{i=0}^{\ell-1} \phi^i(r)$. Note that $R_{\reg}$ and $\cnt(R)$ are stable under $\phi$, so $\hat{r} \in R_{\reg} \cap \cnt(R)$. Furthermore, $\hat{r} \in R^\phi$ and $\hat{r}c=0$. Hence, $c \in \cJ'$.

On the other hand, if $c' \in \cJ'$, then $rc' = 0$ for some $r \in (R^\phi)_{\reg} \cap \cnt(R^\phi)$. If $R$ is a commutative domain, then it is sufficient that $r \neq 0$. If $\phi\restrict{R}=\id_R$, then $R^\phi=R$. In either case, it is clear that $c' \in \cJ \cap D$. 
\end{proof}

Theorem \ref{thm.fix} shows that regularities and part of consistency transfer to the ring of invariants. The next corollary shows that the $\kk$-finitistic property also transfers.

\begin{corollary}\label{cor.kfin}
If $A=A_\mu(R,\sigma,t)$ is a $\kk$-finitistic TGWA of rank $n$, and $B$ is as in the statement of Theorem \ref{thm.fix}, then $B$ is also $\kk$-finitistic.  
\end{corollary}
\begin{proof}
Fix $i,j \in [n]$ with $i \neq j$. By hypothesis, $V_{ij} = \Span_\kk\{\sigma_i^k(t_j) ~|~ k \in \ZZ\}$ has finite $\kk$-dimension. Now define
\begin{align}\label{eq.wij}
W_{ij} = \Span_\kk\{ \tau_i^k(s_j) ~|~ k \in \ZZ\},
\end{align}
which plays the role of $V_{ij}$ for $B$. For $\ell \in \ZZ$, let $V_{ij}^{\sigma_j^\ell} = \Span_\kk\{\sigma_i^k(\sigma_j^\ell(t_j)) ~|~ k \in \ZZ\}$. Since $\sigma_i$ and $\sigma_j$ commute and because $\sigma_j$ is an automorphism, then $\dim_{\kk} V_{ij}^{\sigma_j^\ell} = \dim_{\kk} V_{ij}$. For all $k\in\ZZ$, we can regard $\tau_i^k(s_j)$ as an element of a quotient of $V_{ij} \tensor V_{ij}^{\sigma_j} \tensor \cdots \tensor V_{ij}^{\sigma_j^{m_j-1}}$. It follows that $\dim_{\kk} W_{ij} \leq \left(\dim_{\kk} V_{ij}\right)^{m_j} < \infty$.
\end{proof}

Our final two results demonstrate consequences in the case that $B=A^\phi$. While this is not true in general, in the next two sections we present families of examples where this does hold. First we give a simplicity result for fixed rings of TGWAs. 

Recall that an automorphism $\rho$ of an algebra $A$ is \emph{inner} if there exists a unit $u \in A$ such that $\rho(a)=uau\inv$ for all $a \in A$. A subgroup $G$ of $\Aut(A)$ is called \emph{outer} if the identity is the only inner automorphism of $G$. The hypothesis that $\grp{\phi}$ is outer is rather mild. In our standard examples, the group of units of $A$ is simply $\kk^\times$, in which case the identity is the only inner automorphism in $\Aut(A)$.

\begin{proposition}\label{prop.simp}
Let $A=A_\mu(R,\sigma,t)$, $\phi \in \Aut(A)$ satisfying Hypothesis \ref{hyp.auto}, and $B$ as in the statement of Theorem \ref{thm.fix}. If $A$ is simple, $A^\phi = B$, and $\grp{\phi}$ is outer, then $A^\phi$ is simple.
\end{proposition}
\begin{proof}
This follows directly from \cite[Corollary 2.6]{Mo}.
\end{proof}

Finally we use work in the previous section to study invariant rings of tensor products of TGWAs.

\begin{proposition}\label{prop.tensor}
For $i=1,\hdots,d$, let $A(i)$ be a $\mu$-consistent TGWA and let $\phi_i$ be an automorphism of $A(i)$ satisfying Hypothesis \ref{hyp.auto}. Set $A=\bigotimes_{i=1}^d A(i)$ and $\phi = \phi_1 \tensor \cdots \tensor \phi_d$. If $\gcd(\ord(\phi_i),\ord(\phi_j))=1$ for all $i \neq j$, then $A^\phi=\bigotimes_{i=1}^d A(i)^{\phi_i}$. Consequently, if $A(i)^{\phi_i}$ is a regular, consistent TGWA for each $i$, then so is $A^\phi$.
\end{proposition}
\begin{proof}
This follows directly from Theorem \ref{thm.tensor}.
\end{proof}

\section{$\kk$-finitistic TGWAs of type $(A_1)^n$}\label{sec.A1}

In this section we apply results from the previous section to ($\kk$-finitistic) TGWAs of type $(A_1)^n$. That is, the matrix $C_A=2 I_n$ is a Cartan matrix of type $(A_1)^n$. This family includes GWAs of all finite rank, as well as the quantized Weyl algebras from Example \ref{ex.mqwa}.

Let $A=A_\mu(R,\sigma,t)$ be a TGWA of type $(A_1)^n$. Recall that we assume $A$ is regular and $\mu$-consistent throughout.
Then for all $i \neq j$, $p_{ij}(x)=x-\gamma_{ij}$ for some $\gamma_{ij} \in \kk^\times$.
This condition is equivalent to 
\begin{align}\label{eq.sigma}
\sigma_i(t_j) = \gamma_{ij} t_j \quad\text{for all $i \neq j$.}
\end{align}

\begin{theorem}[Futorny, Hartwig {\cite[Theorem 4.1]{FH1}}]\label{thm.finit}
Let $A=A_\mu(R,\sigma,t)$ be a TGWA of type $(A_1)^n$ with the $\gamma_{ij}$ as above. Then $A$ is (isomorphic to) the $\kk$-algebra generated over $R$ by $X_1^{\pm},\hdots,X_n^{\pm}$ with relations
\begin{align*}
&X_i^{\pm} r - \sigma_i^{\pm 1}(r) X_i^{\pm}, \quad
X_i^-X_i^+ - t_i, \quad X_i^+X_i^- - \sigma_i(t_i),
&\text{for all $r \in R$ and $i \in [n]$,} \\
&X_i^+ X_j^- - \mu_{ij} X_j^- X_i^+, \quad
X_i^+ X_j^+ - \gamma_{ij}\mu_{ij}\inv X_j^+ X_i^+, \quad
X_j^- X_i^- - \gamma_{ij} \mu_{ji}\inv X_i^- X_j^-,
 &\text{for all $i,j \in [n]$, $i \neq j$.}
\end{align*}
\end{theorem}

\begin{example}
A TGWA of type $(A_1)^n$ is a GWA (see Example \ref{ex.gwa}) if $\mu_{ij}=\gamma_{ij}=1$ for all $i,j$.
\end{example}

\begin{example}
Let $A$ and $A'$ be TGWAs of type $(A_1)^n$ and $(A_1)^m$, respectively. By Corollary \ref{cor.tensorprops}, $A \tensor A'$ is of type $(A_1)^{n+m}$. 
\end{example}

We now show that the class of TGWAs of type $(A_1)^n$ is stable under taking fixed rings by the automorphisms in Hypothesis \ref{hyp.auto}. The following is a more general version of \cite[Theorem 3.3]{GHo}, which itself is a generalization of \cite[Theorem 2.6]{JW}.

\begin{theorem}\label{thm.fix1}
Let $R$ be a domain and let $A=A_\mu(R,\sigma,t)$ be a TGWA of type $(A_1)^n$. Let $\phi$ be an automorphism of $A$ satisfying Hypothesis \ref{hyp.auto}. Then $A^\phi=A_\nu(R^\phi,\tau,s)$ is a TGWA of type $(A_1)^n$.
\end{theorem}
\begin{proof}
Let $B$ be the subalgebra of $A^\phi$ generated over $R^\phi$ by the $(X_i^{\pm})^{m_i}$. By Theorem \ref{thm.fix}, $B$ is a TGWA of rank $n$. We claim that $B=A^\phi$ and that $B$ is of type $(A_1)^n$.

Since $\phi$ preserves the space of reduced monomials (as defined in Section \ref{sec.tgwa}), it suffices to determine which reduced monomials are fixed by $A$. Moreover, $\phi$ preserves the standard $\ZZ$-grading on $A$. Hence, $(A_g)^\phi = (A^\phi)_g$ for all $g \in \ZZ^n$. It is immediately clear that $(A^\phi)_\bzero = R^\phi$.

Let $p \in A$ be a reduced monomial. Write
\begin{align}
\label{eq.fixp}
p=X_{i_1}^- \cdots X_{i_u}^- X_{j_1}^+ \cdots X_{j_v}^+.
\end{align}
Then $\{i_1,\hdots,i_u\} \cap \{j_1,\hdots,j_v\} = \emptyset$ by \eqref{eq.reduced} and using the additional relations from \eqref{thm.finit}, we may also assume that 
$i_1 \leq i_2 \leq \cdots \leq i_u$ and $j_1 \leq j_2 \leq \cdots \leq j_v$. 
We have 
\[ \phi(p)=\left( \alpha_{i_1}\inv \cdots \alpha_{i_u}\inv \alpha_{j_1} \cdots \alpha_{j_v} \right) p.\]
Suppose $p \in A^\phi$. By condition (3) of Hypothesis \ref{hyp.auto}, $m_f$ divides the number of factors of $\alpha_f^{\pm}$ that appear in the above coefficient for each index $f$. It follows that $p$ is the product of reduced monomials in the $(X_i^{\pm})^{m_i}$.

More generally, choose $rp \in A^\phi$ with $r \in R$ and $p$ as in \eqref{eq.fixp}. Denote by $\orb(r)$ the cardinality of the orbit of $r$ under the action of $\phi\restrict{R}$. Then we have
\[ rp = \phi(rp)=\left( \alpha_{i_1}\inv \cdots \alpha_{i_u}\inv \alpha_{j_1} \cdots \alpha_{j_v} \right) \phi(r) p.\]
Because $R$ is a regular GWA, then $\phi(r) = \left( \alpha_{i_1}\inv \cdots \alpha_{i_u}\inv \alpha_{j_1} \cdots \alpha_{j_v} \right)\inv r$. Then $\phi^{m_{i_1} \cdots m_{i_u} m_{j_1} \cdots m_{j_v}}(r) = r$, so $\orb(r)$ divides $m_{i_1} \cdots m_{i_u} m_{j_1} \cdots m_{j_v}$. But $\orb(r)$ also divides $\ell$, so $\orb(r)=1$ and $r\in R^\phi$. It follows that $A^\phi$ is generated over $R^\phi$ by the $(X_i^{\pm})^{m_i}$, i.e., $B=A^\phi$.

That $A^\phi$ is in fact a TGWA of type $(A_1)^n$ follows directly from the proof of Corollary \ref{cor.kfin}. It also is easy to check directly. Suppose $i \neq j$. Since $\sigma_i(t_j)=\gamma_{ij} t_j$ and $\sigma_i,\sigma_j$ commute, then
\begin{align*}
\sigma_i^{m_i}(s_j) 
	&= \sigma_i^{m_i}\left( \prod_{k=0}^{m_j-1} \sigma_j^{-k}(t_j) \right)
	= \prod_{k=0}^{m_j-1} \sigma_i^{m_i}\left( \sigma_j^{-k}(t_j) \right)
	= \prod_{k=0}^{m_j-1} \sigma_j^{-k} \left( \sigma_i^{m_i}(t_j) \right) \\
	&= \prod_{k=0}^{m_j-1} \sigma_j^{-k} \left(\gamma_{ij}^{m_i}t_j \right)
	= \prod_{k=0}^{m_j-1} \sigma_j^{-k}(t_j)
	= \gamma_{ij}^{m_im_j} s_j.
\end{align*}
This verifies \eqref{eq.sigma} for $A^\phi$ and so $A^\phi$ is of type $(A_1)^n$.

Since $A$ is $\mu$-consistent, then $A^\phi$ satisfies \eqref{eq.cons1} with respect to $\nu$. The second consistency equation, \eqref{eq.cons2}, is automatic in this case. Thus, $A^\phi$ is $\nu$-consistent.
\end{proof}

\begin{remark}
Assume $A$ and $\phi$ satisfy the setup of Theorem \ref{thm.fix1}. The relations of $A^\phi$ follow from Theorem \ref{thm.finit}. Similarly, the $\nu$-consistency of $A^\phi$ follows from Lemma \ref{lem.nucons}.
\end{remark}



\begin{example}
Let $A$ be the $n$th Weyl algebra (see Example \ref{ex.gwa}). Let $\phi \in \Aut(A)$ be a nontrivial automorphism satisfying Hypothesis \ref{hyp.auto} with $\phi\restrict{R} = \id_R$.
By \cite[Theorem 1]{TIK}, $A^\phi \niso A$. However, it is still a GWA since $\nu_{ij} = \gamma_{ij}^{m_im_j}=1$ for all $i,j$ and the $s_i$ are all regular in $R$.

More generally, if $A$ is a GWA then the fixed ring $A^\phi$ is again a GWA. In the next example we consider a special family of GWAs.
\end{example}

\begin{example}
We say a GWA $A=R(\sigma,t)$ of degree $n$ is \emph{quantum} if $R=\kk[h_1,\hdots,h_m]$, or $R=\kk[h_1^{\pm 1},\hdots,h_m^{\pm 1}]$, and $\sigma_i(h_j)=q_{ij} h_j$ for some $q_{ij} \in \kk^\times$. Suppose $\phi$ is as in Hypothesis \ref{hyp.auto} and $R^{\phi}$ is again a (Laurent) polynomial ring, then $A^\phi$ is again a quantum GWA.
\end{example}

\subsection{Ring-theoretic properties}

In this subsection, we consider the question of whether certain ring-theoretic properties pass to the ring of invariants. This includes the noetherian and Auslander-Gorenstein conditions (Corollary \ref{cor.props}). Note that simplicity of the invariant ring can be derived from Proposition \ref{prop.simp}. Furthermore, a simplicity criterion for TGWAs of type $(A_1)^n$ has been established by Hartiwg and \"{O}inert \cite{HO}.

If $R$ is a $\kk$-algebra and $\sigma \in \Aut(R)$, then a \emph{$\sigma$-derivation} is a $\kk$-linear map $\delta$ satisfying $\delta(rs)=\sigma(r)\delta(s) + \delta(s)r$ for all $r,s \in R$. The \emph{Ore extension} $R[x;\sigma,\delta]$ is defined as the $\kk$-algebra generated over $R$ with indeterminate $x$ and relations $xr=\sigma(r)x+\delta(r)$ for all $r \in R$. When $\delta=0$ we write $R[x;\sigma]$. It has been observed that every degree one GWA is a quotient of an Ore extension \cite{bavgldim}. We present a similar construction for TGWAs of type $(A_1)^n$. This will be useful in establishing many properties for TGWAs of this type, including the noetherian property. We then consider how certain properties pass to their invariant rings.

\begin{lemma}\label{lem.ore}
Let $A=A_\mu(R,\sigma,t)$ be a TGWA of type $(A_1)^n$. Then $A$ is a quotient of an iterated Ore extension.
\end{lemma}
\begin{proof}
Let $T_0 = R[X_1^+;\sigma_1]\cdots[X_n^+;\sigma_n]$ where $\sigma_i$ is extended to $R[X_1^+;\sigma_1]\cdots[X_{i-1}^+;\sigma_{i-1}]$ by setting $\sigma_i(X_j^+) = \gamma_{ij}\mu_{ij}\inv X_j^+$ for all $j<i$. Now let $T = T_0[X_1^-;\theta_1,\delta_1]\cdots [X_n^-;\theta_n,\delta_n]$ where, for all $i$, $\theta_i(r)=\sigma_i\inv(r)$ for all $r \in R$,
$\theta_i(X_j^-) = \gamma_{ij}\inv\mu_{ji} X_j^-$ and $\delta_i(X_j^-)=0$ for all $j<i$, and
\begin{align*}
    \theta_i(X_j^+) = \begin{cases}
        X_i^+ & \text{if $i=j$} \\
        \mu_{ij}\inv X_j^+ & \text{if $i\neq j$,}
    \end{cases} \qquad
    \delta_i(X_j^+) = \begin{cases}
        t_i-\sigma_i(t_i) & \text{if $i=j$} \\
        0 & \text{if $i\neq j$.}
    \end{cases}
\end{align*}
Let $I$ be the ideal of $T$ generated by the $X_i^-X_i^+-t_i$ for all $i$. By comparing the relations of $T$ and $A$, we see readily that $A \iso T/I$.
\end{proof}

Before the next result, we review some definitions. A noetherian ring $A$ of finite injective dimension is called \emph{Auslander-Gorenstein (AG)} if for any (left or right) module $M$ and submodule $N$ of $\Ext_A^s(M,A)$, $s \in \ZZ_+$, we have $\Ext_A^i(N,A)=0$ for $i<s$. Hodges proved that every degree one GWA over $\kk[t]$ with $\sigma(t)=t-1$ is AG \cite[Theorem 2.2]{H1}. More generally, Kirkman and Kuzmanovich proved that any degree one GWA over a noetherian AG ring is again AG \cite[Proposition 2.2]{KK}. We denote by $\lgld A$ (resp. $\rgld A$) the left (resp. right) global dimension of $A$. Recall that if $A$ is noetherian, then $\lgld A = \rgld A$ and we refer to this simply as the global dimension of $A$, denoted $\gldim A$.

For a TGWA $A=A_\mu(R,\sigma,t)$ of type $(A_1)^n$, we set $A^{(0)}=R$ and for $k$, $1 \leq k \leq n$, let $A^{(k)}$ be the subalgebra of $A$ generated by $R$ and $X_1^{\pm},\hdots,X_k^{\pm}$. Hence, $A=A^{(n)}$. That $A^{(k)}$ is indeed a subalgebra follows directly from the $\ZZ^n$-grading on $A$. It is clear that $A^{(k)}$ is a TGWA of rank $k$ over $R$ 

\begin{theorem}\label{thm.ore}
Let $A=A_\mu(R,\sigma,t)$ be a TGWA of type $(A_1)^n$.
\begin{enumerate}
\item \label{ore1} If $R$ is (left/right) noetherian, then $A$ is (left/right) noetherian.
\item \label{ore2} If $R$ is an AG domain, then $A$ is an AG domain.
\item \label{ore3} We have $\GKdim(A) \geq \GKdim(R)+n$.
\item \label{ore4} If $\lgld R < \infty$ and $\lgld A < \infty$, then $\lgld R \leq \lgld A \leq \lgld R + n$.
\end{enumerate}
\end{theorem}
\begin{proof}
We keep the notation of Lemma \ref{lem.ore}.

\eqref{ore1} Suppose $R$ is left noetherian. Then $T$ is left noetherian by the skew Hilbert Basis Theorem (see \cite{GW}) and so $A$ is left noetherian. The right noetherian proof is identical.

\eqref{ore2} Suppose $R$ is an AG domain. Since $R$ is a domain, so are $T$ and $A$. By \cite[Theorem 2.2]{Ek1}, $T$ is AG. For all $i,j \in [n]$ with $i \neq j$, 
\begin{align*}
X_i^{\pm}(X_j^-X_j^+-t_j)
    &= \left( \gamma_{ij}^{\pm 1} (X_j^-X_j^+)- \sigma_i^{\pm 1}(t_j) \right) X_i^{\pm}
    = \gamma_{ij}^{\pm 1} (X_j^-X_j^+-t_j)X_i^{\pm}, \\
X_j^-(X_j^-X_j^+-t_j)
        &= X_j^-(X_j^+X_j^- - \sigma_j(t_j))
        = (X_j^-X_j^+-t_j)X_j^-, \\
X_j^+(X_j^-X_j^+-t_j)
    &= (X_j^+X_j^- - \sigma_j(t_j))X_j^+
    = (X_j^-X_j^+-t_j)X_j^+.
\end{align*}
It follows that $z_j=X_j^-X_j^+-t_j$ is a normal element in $T$ for every $j$ and $A\iso T/I$. Moreover, it is clear that $T/(z_1)$ is an Ore extension over $A^{(1)}$. Consequently, $T/(z_1)$ is a domain. More generally, $T/(z_1,\hdots,z_k)$ is an Ore extension over $A^{(k)}$ and also a domain. Thus, $(z_1,\hdots,z_n)$ is a regular normal sequence in $T$. The result now follows from \cite[Proposition 2.1]{ASZ2}.

\eqref{ore3} Since $A$ contains a copy of the iterated skew polynomial ring $R[X_1^+;\sigma_1]\cdots[X_n^+;\sigma_n]$, then the result follows from \cite[Lemma 3.4]{KL}.

\eqref{ore4} Since $A$ is a free $R$-module with basis the reduced monomials, then $\lgld R \leq \lgld A$.
Now suppose $\lgld R < \infty$ and $\lgld A < \infty$.
If $A$ has rank $1$, then $A$ is a GWA and the result follows from \cite[Theorem 2.7]{bavgldim}. We proceed inductively. Suppose $A$ has rank $k+1$ and for all TGWA of rank $k$, $1 \leq k < n$, the upper bound on left global dimension is $\lgld R + k$.

It is clear that $A^{(k+1)}$ is a free module over $A^{(k)}$. Set 
\[S=A^{(k)}[X_{k+1}^+;\sigma_{k+1}][X_{k+1}^-;\theta_{k+1},\delta_{k+1}].\]
Then $A \iso S/(X_{k+1}^-X_{k+1}^+-t_{k+1})$.
By \cite[Theorem 7.5.3(i)]{MR} and the induction hypothesis, we have
\[ \lgld S \leq \lgld A^{(k)} + 2 \leq \lgld R + (k+2).\]
By the above, $X_{k+1}^-X_{k+1}^+-t_{k+1}$ is a normal regular element in $S$. Since $A$ has finite global dimension, then by \cite[Theorem 7.3.5]{MR}, 
\[ \lgld A \leq \lgld S-1 \leq \lgld R + (k+1).\] 
The result now follows by induction.
\end{proof}

\begin{corollary}
\label{cor.props}
Let $A=A_\mu(R,\sigma,t)$ be a TGWA of type $(A_1)^n$. Suppose $\phi$ satisfies Hypothesis \ref{hyp.auto}.
\begin{enumerate}
\item If $R$ is (left/right) noetherian, then so is $A^\phi$.
\item If $R^\phi$ is an AG domain, then so is $A^\phi$.
\end{enumerate}
\end{corollary}
\begin{proof}
By Theorem \ref{thm.ore}\eqref{ore1}, $A$ is (left/right) noetherian. So, (1) follows from \cite[Corollary 1.12]{Mo}. By Theorem \ref{thm.fix1}, $A^\phi$ is a TGWA of type $(A_1)^n$. Hence, (2) follows from Theorem \ref{thm.ore}\eqref{ore2}.
\end{proof}

For Corollary \ref{cor.props} (2), the conditions on $R^\phi$ simplify depending upon the additional hypothesis on $R$. If $\phi$ is the identity on $R$, then it suffices to assume that $R$ is an AG domain. If $R$ is commutative, the AG condition simplifies to the Gorenstein condition. As in many examples, if $R$ is a polynomial ring, then one can apply Watanabe's Theorem \cite{watanabe} to determine whether $R^\phi$ is Gorenstein.

\section{A rank $2$ case}\label{sec.A2}

Let $A$ be a TGWA of type $A_2$. Recall we assume $A$ is regular and $\mu$-consistent. Then the rank of $A$ is $n=2$ and $\deg p_{12}=\deg p_{21}=2$, where $p_{12}$ and $p_{21}$ are the minimal polynomials as in Definition \ref{def:k-finite}. For the rest of this section we fix $\lambda_1,\lambda_2,\eta_1,\eta_2\in\kk$ such that 
\begin{equation}\label{eq.p12p21}
p_{12}(x)=x^2+\lambda_1x+\lambda_2,\qquad p_{21}(x)=x^2+\eta_1x+\eta_2.
\end{equation}
By \cite[Theorem 3.2(b)]{FH1} and \cite[Theorem 6.1]{hart3}, $A$ is generated by $R$, $X_1^\pm$, $X_2^\pm$ with a complete list of relations being \eqref{eq.tgwc1}-\eqref{eq.tgwc3} plus
\begin{align}
\label{eq.A2rel1}
&(X_1^{+})^2X_2^{+}+\lambda_1\mu_{12}\inv X_1^{+} X_2^{+}X_1^{+} +\lambda_2 \mu_{12}^{-2} X_2^{+}(X_1^{+})^2=0, \\
\label{eq.A2rel2} &X_2^{-}(X_1^{-})^2+\lambda_1\mu_{21}\inv X_1^{-} X_2^{-}X_1^{-} +\lambda_2 \mu_{21}^{-2} (X_1^{-})^2X_2^{-}=0, \\
\label{eq.A2rel3}
&(X_2^{+})^2X_1^{+}+\eta_1\mu_{21}\inv X_2^{+} X_1^{+}X_2^{+} +\eta_2 \mu_{21}^{-2} X_1^{+}(X_2^{+})^2=0, \\
\label{eq.A2rel4} &X_1^{-}(X_2^{-})^2+\eta_1\mu_{12}\inv X_2^{-} X_1^{-}X_2^{-} +\eta_2 \mu_{12}^{-2} (X_2^{-})^2X_1^{-}=0.
\end{align}
Hence, this is a generalization of Example \ref{ex.A2simple}.

\begin{example}\label{ex.A2}
Let $R=\kk[h]$, let $p \in \kk^\times$, and let $\beta \in \kk$. Assume $(p,\beta) \neq (1,0)$. We allow $\mu=\mu_{12} \in \kk^\times$ but require $\mu_{21}=\mu\inv$. Define $\sigma_1,\sigma_2 \in \Aut R$ by $\sigma_1(h)=ph+\beta$ and $\sigma_2(h)=p\inv(h-\beta)$. Let $t_1=h$ and $t_2=ph+\beta$. Then $\sigma_1\sigma_2=\id_R$, $\sigma_1(t_1)=t_2$ and $\sigma_2(t_2)=t_1$. Thus, for any $\mu$, $A=\cA_\mu(R,\sigma,t)$ is a TGWA of type $A_2$.

We have
\[ \sigma_2^2(t_1) - (1+p\inv)\sigma_2(t_1) + p\inv t_1 
= (p^{-2}h - \beta(p^{-2}+p\inv)) - (1+p\inv)p\inv(h-\beta) + p\inv h = 0.\]
Thus, the minimal polynomial for $\sigma_2$ on $V_{21}$ is $x^2-(p\inv+1)x+p\inv=(x-p\inv)(x-1)$. 
Similarly, the minimal polynomial for $\sigma_1$ on $V_{12}$ is $x^2-(p+1)x+p=(x-p)(x-1)$.

Note that if $\beta \neq 0$, then by making the change of variable $h \mapsto \beta h$ and $X_2^+ \mapsto \beta X_2^+$, we may assume $\beta=1$.
\end{example}

\begin{remark}\label{rem.klein}
Recall the TGWAs of Example \ref{ex.rank2}, with $R=\kk[h]$ and $\sigma_i(h)=h-\beta_i$, $i=1,2$. Notice that, among those, the only ones that are of type $A_2$ are those such that $t=(p_1,p_2)$ with $\deg p_1=\deg p_2=1$. This is because, if say $\deg p_2=k$, then the $\kk$-vector space $V_{12}=\kk\{\sigma_1^\ell(p_2)~|~\ell\in\mathbb{Z}\}$ has a basis $\{1,h,h^2,\ldots,h^k\}$ and the minimal polynomial of $\sigma_1$ on $V_{12}$ is $(x-1)^{k+1}$. The same happens if we replace $p_2$ with $p_1$ and $\sigma_1$ with $\sigma_2$. It then follows that the noncommutative Kleinian fiber products are of type $A_2$ if and only if $\deg p_1=\deg p_2=1$. But this implies that $\frac{\beta_1}{\beta_2}=-1$, so these are essentially the algebras of Example \ref{ex.A2}.
\end{remark}

\begin{example}\label{ex.A2b}
Here we present examples of type $A_2$ TGWAs over $R=\kk[h_1,h_2]$.
\begin{enumerate}
\item This is a specific example from the family studied by Sergeev \cite[Section 1.5.2]{sergeev}. Set $\sigma_i(h_j)=h_j-\delta_{ij}$ for $i,j \in \{1,2\}$. Set $t_1=h_2-h_1$ and $t_2=h_2-h_1+1$. By \cite[Proposition 8.6]{HO}, $\cA(R,\sigma,t)$ is simple. 

\item (\cite[Example 1]{MT})
Set $\sigma_1(h_i)=h_i+1$ for $i=1,2$, $\sigma_2(h_1) = h_1-1$, and $\sigma_2(h_2)=h_2$.
Let $t_1=h_1h_2$ and $t_2=h_1+1$. In this case, and in case (1), we have $p_{12}(x)=p_{21}(x)=x^2-2x+1$.


\item (\cite[Section 5]{hart3}) 
Let $\mu,q \in \kk^\times$. Define $\sigma_1,\sigma_2 \in \Aut(R)$ by
\begin{align*}
&\sigma_1(h_1) = \mu q\inv h_1-h_2, &
&\sigma_1(h_2) = \mu q h_2, \\
&\sigma_2(h_1) = \mu q h_1 + h_2,  & 
&\sigma_2(h_2) = \mu q\inv h_2.
\end{align*}
Set $t_1 = h_1$ and $t_2 = \mu\inv q\inv h_1 - \mu^{-2}h_2$. Then $p_{12}=p_{21}=x^2-\mu(q+q\inv) x + \mu^{-2}$.
\end{enumerate}
\end{example}

\begin{definition}\label{def.cheb-poly}
Let $q$ and $\beta$ be indeterminates and let $a\in\NN$. We define polynomials $S_a(q,\beta)\in\ZZ[q,\beta]$ by
\[S_a(q,\beta):=\sum_{i=0}^{\lfloor\frac{a}{2}\rfloor}(-1)^i\binom{a-i}{i}\beta^i q^{a-2i}.\]
\end{definition}

\begin{remark}\label{rem.recur}
The polynomials $S_a=S_a(q,\beta)$ can also be defined by the recurrence relation 
\[ S_{a+1}=q S_a-\beta S_{a-1}, \qquad S_0=1, \quad S_1=q.\]
They are related to the Chebyshev polynomials of the second kind by 
\begin{equation}\label{eq.chebyshev}
S_a(q,\beta)=\beta^{a/2} U_a\left(\frac{q}{2\beta^{1/2}}\right)
\end{equation}
and they satisfy the identity
\begin{equation}\label{eq.cheby}
\beta S_{c-2}S_{a-1}+S_{a+c-1}=S_aS_{c-1}
\end{equation}
for all $a\geq 1$, $c\geq 2$.
\end{remark}

\begin{lemma}\label{lem.rewrite}Suppose $x,y$ are elements in a $\kk$-associative algebra $A$ such that for some $q,\beta,q',\beta'\in\kk$, $q,\beta \neq 0$, the relations
\begin{align}
\label{eq.rel-lem} x^2y-qxyx+\beta yx^2& =0 \\
\label{eq.rel-lem2}y^2x-q'yxy+\beta'xy^2 &=0
\end{align}
are satisfied, and $S_a(q,\beta)\neq 0$ for all $a\geq 0$. Then for all $a,b,c\geq 0$ we have that, in $A$
\begin{equation}\label{eq.lem-rewrite} x^ay^bx^c=\sum_{i=0}^b \gamma_i y^i x^{a+c}y^{b-i} \end{equation}
for some $\gamma_i\in\kk$, $0\leq i\leq b$.
\end{lemma}
\begin{proof}
If any of $a$, $b$ or $c$ equals $0$, then the statement is clearly true, so we can assume that $a,b,c\geq 1$. Now suppose that $b=c=1$. Then we claim that
\begin{equation}\label{eq.lem1}
x^ayx=\frac{S_{a-1}(q,\beta)}{S_a(q,\beta)}x^{a+1}y+\frac{\beta^a}{S_a(q,\beta)}yx^{a+1}.
\end{equation}

We proceed by induction on $a$. For $a=1$ we have, by \eqref{eq.rel-lem},
\[xyx=\frac{1}{q}x^2y+\frac{\beta}{q}yx^2=\frac{S_0}{S_1}x^2y+\frac{\beta^1}{S_1}yx^2.\]
For $a\geq 2$, we have
\begin{align*}
x^ayx \quad
    &\stackrel{\mathclap{\eqref{eq.rel-lem}}}{=} \quad
    x^{a-1}\left(\frac{1}{q}x^2y+\frac{\beta}{q}yx^2\right)
    = \frac{1}{q}x^{a+1}y+\frac{\beta}{q}x^{a-1}yx^2 \\
   &\stackrel{\mathclap{\text{(ind hyp)}}}{=} \quad \frac{1}{q}x^{a+1}y+\frac{\beta}{q}\left(\frac{S_{a-2}}{S_{a-1}}x^ay+\frac{\beta^{a-1}}{S_{a-1}}yx^a\right)x 
    = \frac{1}{q}x^{a+1}y+\frac{\beta S_{a-2}}{qS_{a-1}}x^ayx+\frac{\beta^a}{qS_{a-1}}yx^{a+1}.
\end{align*}
Now by Remark \ref{rem.recur},
\begin{align*}
\frac{qS_{a-1}-\beta S_{a-2}}{qS_{a-1}}x^ayx&=\frac{1}{q}x^{a+1}y+\frac{\beta^a}{qS_{a-1}}yx^{a+1}\\
\frac{S_a}{qS_{a-1}}x^ayx&=\frac{1}{q}x^{a+1}y+\frac{\beta^a}{qS_{a-1}}\\
x^ayx&=\frac{S_{a-1}}{S_a}x^{a+1}y+\frac{\beta^a}{S_a}yx^{a+1}.
\end{align*}

Now we claim that for all $a,c\geq 1$ we have
\begin{equation}\label{eq.lem2}
x^ayx^c=\frac{S_{a-1}(q,\beta)}{S_{a+c-1}(q,\beta)}x^{a+c}y+\frac{\beta^a S_{c-1}(q,\beta)}{S_{a+c-1}(q,\beta)}yx^{a+c}.
\end{equation}

The base case $c=1$ is true from \eqref{eq.lem1}. Now suppose $c\geq 2$, then
\begin{align*}
x^ayx^c \quad
    &\stackrel{\mathclap{\eqref{eq.rel-lem}}}{=} \quad  
    x^{a-1}\left(\frac{1}{q}x^2y+\frac{\beta}{q}yx^2\right)x^{c-1}
    =\frac{1}{q}x^{a+1}yx^{c-1}+\frac{\beta}{q}(x^{a-1}yx)x^c \\
    &\stackrel{\mathclap{\text{(ind hyp)}}}{=} \quad  
    \frac{1}{q}\left(\frac{S_a}{S_{a+c-1}}x^{a+c}y+\frac{\beta^{a+1}S_{c-2}}{S_{a+c-1}}yx^{a+c}\right)+
\frac{\beta}{q}\left(\frac{S_{a-2}}{S_{a-1}}x^{a}y+\frac{\beta^{a-1}}{S_{a-1}}yx^{a}\right)x^c \\
    &= \frac{S_a}{qS_{a+c-1}}x^{a+c}y+\frac{\beta^{a+1}S_{c-2}}{qS_{a+c-1}}yx^{a+c}+
\frac{\beta S_{a-2}}{qS_{a-1}}x^ayx^c+\frac{\beta^{a}}{qS_{a-1}}yx^{a+c}.
\end{align*}
Then by Remark \ref{rem.recur} and \eqref{eq.cheby} we have
\begin{align*}
\frac{qS_{a-1}-\beta S_{a-2}}{q S_{a-1}}x^ayx^c &= \frac{S_a}{qS_{a+c-1}}x^{a+c}y+\frac{\beta^{a+1}S_{c-2}S_{a-1}+\beta^aS_{a+c-1}}{qS_{a-1}S_{a+c-1}}yx^{a+c}\\
\frac{S_a}{qS_{a-1}}x^ayx^c &= \frac{S_a}{qS_{a+c-1}}x^{a+c}y+\beta^a\frac{\beta S_{c-2}S_{a-1}+S_{a+c-1}}{qS_{a-1}S_{a+c-1}}yx^{a+c}\\
\frac{S_a}{qS_{a-1}}x^ayx^c &= \frac{S_a}{qS_{a+c-1}}x^{a+c}y+\beta^a\frac{S_aS_{c-1}}{qS_{a-1}S_{a+c-1}}yx^{a+c}\\
x^ayx^c&=\frac{S_{a-1}}{S_{a+c-1}}x^{a+c}y+\frac{\beta^a S_{c-1}}{S_{a+c-1}}yx^{a+c}.
\end{align*}

Therefore, \eqref{eq.lem-rewrite} is satisfied when $b=1$. Now we want to prove \eqref{eq.lem-rewrite} for $c=1$ and any $b\geq 1$. We have already shown that \eqref{eq.lem-rewrite} is true for $b=c=1$. We proceed by induction on $b$. If $b\geq 2$ we have
\begin{align*}
x^ay^bx
    &= x^ay^{b-2}(y^2x) \quad 
    \stackrel{\mathclap{\eqref{eq.rel-lem2}}}{=} \quad  
    x^a y^{b-2}(q'yxy-\beta'xy^2)
    = q'(x^ay^{b-1}x)y-\beta'(x^ay^{b-2}x)y^2 \\
    &\stackrel{\mathclap{\text{(ind hyp)}}}{=} \quad 
    (\sum_i \gamma'_i y^i x^{a+1}y^{b-1-i})y-(\sum_i \gamma''_i y^i x^{a+1}y^{b-2-i})y^2
    = \sum_i \gamma_i y^ix^{a+1}y^{b-i}.
\end{align*}

Finally, we can conclude the proof of \eqref{eq.lem-rewrite} by double induction on $b$ and $c$. We have already shown it to be true when $b=1$ for any $c$ and when $c=1$ for any $b$, so now suppose $b,c\geq 2$ and we have
\begin{align*}
x^ay^bx^c \quad
    &\stackrel{\mathclap{\eqref{eq.rel-lem2}}}{=} \quad  
    x^ay^{b-2}(q'yxy-\beta'xy^2)x^{c-1} \\
    &= q'(x^ay^{b-1}x)yx^{c-1}-\beta'(x^ay^{b-2}x)y^2x^{c-1} \\
    &\stackrel{\mathclap{\left(\substack{\text{ind hyp} \\ \text{on $b$}}\right)}}{=} \quad
    \left(\sum_i \zeta'_i y^i x^{a+1}y^{b-1-i}\right)yx^{c-1}-\left(\sum_i \zeta''_i y^i x^{a+1}y^{b-2-i}\right)y^2x^{c-1} \\
    &= \sum_i \zeta_i y^i(x^{a+1}y^{b-i}x^{c-1}) \\
    &\stackrel{\mathclap{\left(\substack{\text{ind hyp} \\ \text{on $c$}}\right)}}{=} \quad
    \sum_i\zeta_iy^i\left(\sum_j\xi_j y^j x^{a+c}y^{b-i-j}\right)\\
    &= \sum_{i,j}\zeta_i\xi_j y^{i+j}x^{a+c}y^{b-i-j}\\
    &=\sum_i \gamma_iy^i x^{a+c}y^{b-i}.\qedhere
\end{align*}
\end{proof}

\begin{lemma}\label{lem.monomial-span}
Let $A$ be a TGWA of type $A_2$ with minimal polynomials as in \eqref{eq.p12p21}. 
\begin{enumerate}
\item \label{span1} If $S_a(-\lambda_1,\lambda_2)\neq 0$ for all $a\geq 0$, then the monomials
\[(X_2^+)^a(X_1^+)^b(X_2^+)^c, \quad (X_1^+)^a(X_2^-)^b, \quad (X_1^-)^a (X_2^+)^b, \quad (X_2^-)^a(X_1^-)^b(X_2^-)^c\]
with $a,b,c\geq 0$ generate $A$ as a left (and as a right) $R$-module.

\item \label{span2} If $S_a(-\eta_1,\eta_2)\neq 0$ for all $a\geq 0$, then the monomials
\[(X_1^+)^a(X_2^+)^b(X_1^+)^c, \quad (X_1^+)^a(X_2^-)^b, \quad (X_1^-)^a (X_2^+)^b, \quad (X_1^-)^a(X_2^-)^b(X_1^-)^c\]
with $a,b,c\geq 0$ generate $A$ as a left (and as a right) $R$-module.
\end{enumerate}
\end{lemma}
\begin{proof}
By Theorem \ref{thm.props} (2), the reduced monomials generate $A$ over $R$ both as a left and a right module. If $S_a(-\lambda_1,\lambda_2)\neq 0$, we have that 
\begin{align*}
S_a(-\lambda_1\mu_{12}\inv,\lambda_2\mu_{12}^{-2}) 
    &=\mu_{12}^{-a}S_a(-\lambda_1,\lambda_2)\neq 0, \\
S_a(-\lambda_1\mu_{21}\inv,\lambda_2\mu_{21}^{-2})
    &=\mu_{21}^{-a}S_a(-\lambda_1,\lambda_2)\neq 0.
\end{align*}
Hence, we can apply Lemma \ref{lem.rewrite} to $x=X_1^+$, $y=X_2^+$ by \eqref{eq.A2rel1} and \eqref{eq.A2rel3} to write all monomials involving only $X_1^+$ and $X_2^+$ in the form $(X_2^+)^a(X_1^+)^b(X_2^+)^c$. We also do the same thing with $x=X_1^-$, $y=X_2^-$ by \eqref{eq.A2rel2} and \eqref{eq.A2rel4} to write all monomials involving only $X_1^-$ and $X_2^-$ in the form $(X_2^-)^a(X_1^-)^b(X_2^-)^c$. This proves \eqref{span1}. The proof of part \eqref{span2} is identical, exchanging the role of $X_1^+$ with $X_2^+$ and the role of $X_1^-$ with $X_2^-$.
\end{proof}

\begin{remark}\label{rem.split}
If $\chr{\kk}=0$ and $p_{12}(x)\in\QQ[x]$, we can give easier conditions that imply that $S_a(-\lambda_1,\lambda_2)\neq 0$ (the following also applies if $p_{12}(x)\in\RR[x]$, other conditions can also be given if the coefficients are in any field $\mathbb{F}$ with $\QQ\subset \mathbb{F}\subset \CC$). 

First of all, notice that $\lambda_2\neq 0$ because $\sigma_1$ is invertible. If $\lambda_2<0$, and $a$ is even (resp. odd) then $S_a(-\lambda_1,\lambda_2)$ is a positive linear combination of even (resp. odd) monomials in $\lambda_1$, including a constant term (no such term when $a$ is odd). It follows that, when $\lambda_2<0$, $S_a(-\lambda_1,\lambda_2)>0$ for all $a$ even, and for $a$ odd we have $S_a(-\lambda_1,\lambda_2)=0$ if and only if $\lambda_1=0$. 

If $\lambda_2>0$, by  \eqref{eq.chebyshev} we have 
\[S_a(-\lambda_1,\lambda_2)=0 \iff \lambda_2^{a/2}U_a\left(\frac{-\lambda_1}{2\lambda_2^{1/2}}\right)=0\iff U_a\left(\frac{-\lambda_1}{2\lambda_2^{1/2}}\right)=0.\]
The roots of the Chebyshev polynomial of second kind $U_a$ over the complex numbers are all real and they are 
\[\left\{\left. \cos\left(\frac{k}{a+1}\pi\right)~\right|~k=1,\ldots,a \right\}\subset (-1,1).\]
In particular, we will have $S_a(-\lambda_1,\lambda_2)\neq 0$ if $\frac{-\lambda_1}{2\lambda_2^{1/2}}\not\in(-1,1)$.
\begin{align*}\left|\frac{-\lambda_1}{2\lambda_2^{1/2}} \right|\geq 1 
\iff \left|-\lambda_1\right| \geq 2\left|\lambda_2\right|^{1/2} 
\iff \left|\lambda_1\right|\geq 2\left|\lambda_2\right|^{1/2} 
\iff \left|\lambda_1\right|^2 \geq 4\left|\lambda_2\right| \\
\iff \lambda_1^2\geq 4 \lambda_2 
\iff \lambda_1^2-4\lambda_2 \geq 0\iff p_{12}(x) \text{ splits over $\RR$.}\end{align*}
In conclusion, if $p_{12}(x)\in\QQ[x]$ splits over $\RR$ (regardless of whether $\RR\subset \kk$), $p_{12}(x)=(x-z_1)(x-z_2)$, with $z_1,z_2\in\RR$, we have $\lambda_1=-(z_1+z_2)$ and $\lambda_2=z_1z_2$. Then $S_a(-\lambda_1,\lambda_2)\neq 0$ for all $a\geq 0$ if $z_1\neq -z_2$.
\end{remark}

\begin{example}
Due to Remark \ref{rem.split}, we have that $S_a(-\lambda_1,\lambda_2)\neq 0\neq S_a(-\eta_1,\eta_2)$ for all $a\geq 0$ in the cases of Example \ref{ex.A2b}. The same also applies to the cases of Example \ref{ex.A2}, as long as $p\neq-1$.
\end{example}

\begin{hypothesis}\label{hyp.A2auto}
Assume $\phi_1$ and $\phi_2$ are automorphisms of $A$ satisfying Hypothesis \ref{hyp.auto}. For $\phi_1$, we assume $\alpha_1=1$ and for $\phi_2$ we assume $\alpha_2=1$.
\end{hypothesis}

\begin{theorem}\label{thm.A2}
Let $A=\cA_\mu(R,\sigma,t)$ be a TGWA of type $A_2$. Assume $\phi_1$ and $\phi_2$ are as in Hypothesis \ref{hyp.A2auto}.
\begin{enumerate}
\item \label{A2-1} If $S_a(-\eta_1,\eta_2)\neq 0$ for all $a\geq 0$, then $A^{\phi_1}$ is a regular, $\nu$-consistent rank $2$ TGWA.
\item \label{A2-2} If $S_a(-\lambda_1,\lambda_2)\neq 0$ for all $a\geq 0$, then $A^{\phi_2}$ is a regular, $\nu$-consistent rank $2$ TGWA.
\end{enumerate}
\end{theorem}

\begin{remark}
Notice that $A^{\phi_1}$ and $A^{\phi_2}$ are in general no longer of type $A_2$. In fact, if $A$ is a Kleinian fiber product as in Example \ref{ex.rank2}, by Remark \ref{rem.klein} it is of type $A_2$ if and only if $t=(p_1,p_2)$ with $\deg p_1=\deg p_2=1$. However $A^{\phi_{3-i}}=\mathcal{A}_\nu(R,\tau,s)$ is again a Kleinian fiber product but with $s_{i}=\prod_{k=0}^{m_i-1} \sigma^{-k}(p_i)$ which is a polynomial of degree $m_i$.
\end{remark}

\begin{proof}[Proof of Theorem \ref{thm.A2}]
We prove \eqref{A2-1}. The proof for \eqref{A2-2} is analogous. Let $B$ be the subalgebra of $A^{\phi_1}$, as defined in Theorem \ref{thm.fix}. We want to show that $A^{\phi_1}=B$, which will prove the result (since $A^{\phi_1}$ is a TGWA of rank $2$, \eqref{eq.cons1} is sufficient for consistency). By Lemma \ref{lem.monomial-span} \eqref{span2}, since $S_a(-\eta_1,\eta_2)\neq 0$ we have that the monomials 
\[(X_1^+)^a(X_2^+)^b(X_1^+)^c, \quad (X_1^+)^a(X_2^-)^b, \quad (X_1^-)^a (X_2^+)^b, \quad (X_1^-)^a(X_2^-)^b(X_1^-)^c\]
span $A$ over $R$. Hence, since $(\ell,m_2)=1$, $A^{\phi_1}$ will be spanned over $R^{\phi_1}$ by the monomials fixed by $\phi_1$. We have
\begin{align*}
   \phi_1 (X_1^+)^a(X_2^+)^b(X_1^+)^c&= \alpha_2^b(X_1^+)^a(X_2^+)^b(X_1^+)^c, \\
   \phi_1 (X_1^+)^a(X_2^-)^b &= \alpha_2^{-b}(X_1^+)^a(X_2^-)^b,\\ 
   \phi_1 (X_1^-)^a (X_2^+)^b &= \alpha_2^b (X_1^-)^a (X_2^+)^b,\\ 
   \phi_1 (X_1^-)^a(X_2^-)^b(X_1^-)^c &= \alpha_2^{-b}(X_1^-)^a(X_2^-)^b(X_1^-)^c,
\end{align*}
so those monomials are fixed if and only $m_2\mid b$, if and only if those are actually monomials in $X_1^\pm$ and $(X_2^\pm)^{m_2}$. It follows that $A^{\phi_1}=B$ as desired.
\end{proof}

\begin{corollary}\label{cor.A2tensor}
For $i=1,\hdots,d$, let $A(i)$ be a TGWA of type $A_2$ and let $\phi_i$ be an automorphism of $A(i)$ satisfying Hypothesis \ref{hyp.A2auto}. Set  $A=\bigotimes_{i=1}^d A(i)$ and $\phi = \phi_1 \tensor \cdots \tensor \phi_d$. If $\gcd(\ord(\phi_i),\ord(\phi_j))=1$ for all $i \neq j$, then $A^\phi=\bigotimes_{i=1}^d A(i)^{\phi_i}$ is a $\kk$-finitisitic TGWA.
\end{corollary}
\begin{proof}
This follows from Corollary \ref{cor.tensorprops}, Proposition \ref{prop.tensor}, and Theorem \ref{thm.A2}.
\end{proof}

It readily follows from the above that one could replace one of the $A(i)$ in Corollary \ref{cor.A2tensor} with a TGWA of type $(A_1)^n$ and an appropriate automorphism $\phi_i$ satisfying Hypothesis \ref{hyp.auto}.

\subsection{The noetherian property}

It is unknown whether rank 2 TGWAs, or even those of type $A_2$, over a noetherian base ring are noetherian. Here we show that a certain family of rank 2 TGWAs, obtained through taking invariants, satisfies this property. The method is related to the down-up algebras of Benkart and Roby \cite{BRdu}, and uses the following technique which is well-known.

\begin{remark}\label{rmk.hberg}
The enveloping algebra of the Heisenberg Lie algebra can be presented as 
\[T = \kk\langle x,y,z ~|~ xz=zx, yz=zy, xy-yx=z\rangle.\]
We can realize $T$ as the skew polynomial ring $\kk[y,z][x;\delta]$ where $\delta(y)=z$ and $\delta(z)=0$. Substituting $z=xy-yx$ into the relations $0=xz-zx=yz-zy$ gives
\begin{align*}
0 &= xz-zx = x(xy-yx)-(xy-yx)x = x^2y-2xyx+yx^2 \\
0 &= yz-zy = y(xy-yx)-(xy-yx)y = -(y^2x-2yxy+xy^2).
\end{align*}
Hence, we can alternatively present $T$ as a quotient of the free algebra $\kk\langle x,y \rangle$ modulo these two relations. Here we generalize to the case of TGWAs of type $A_2$ presented in Example \ref{ex.A2}.
\end{remark}

\begin{theorem}\label{thm.ore2}
Let $A$ be the TGWA in Example \ref{ex.A2} and assume $\mu=\mu_{12}=1$ and $p=1$. Then $A$  is a quotient of an iterated Ore extension and so $A$ is (left/right) noetherian.
\end{theorem}
\begin{proof}
In this case, the generators of the canonical graded ideal are
\begin{align*}
&(X_1^{\pm})^2X_2^{\pm}-2 X_1^{\pm}X_2^{\pm}X_1^{\pm} + X_2^{\pm}(X_1^{\pm})^2, \\
&(X_2^{\pm})^2X_1^{\pm}-2 X_2^{\pm}X_1^{\pm}X_2^{\pm}+X_1^{\pm}(X_2^{\pm})^2.
\end{align*}

Let $z$ be a new central indeterminate and set $R'=R[z]$. Set $S$ to be the $\kk$-algebra generated over $R'$ by $X_1^-$ and $X_2^-$ satisfying the approprite Serre relations above. Then $S$ is a down-up algebra which is noetherian by \cite{KMP}. We will build an iterated Ore extension over $S$.

Extend the automorphism $\sigma_1$ to all of $S$ by $\sigma_1(z)=z$, $\sigma_1(X_2^-)=X_2^-$, and $\sigma_1(X_1^-)=X_1^-$. Define a $\sigma_1$-derivation $\delta_1$ of $S$ by $\delta_1(r)=\delta_1(z)=\delta_1(X_2^-)=0$ and $\delta_1(X_1^-)=\sigma_1(t_1)-t_1=1$. Set $T_1=S[X_1^+;\sigma_1,\delta_1]$.

Extend the automorphism $\sigma_2$ to all of $T_1$ by $\sigma_2(z)=z$, $\sigma_2(X_1^-) = X_1^-$, $\sigma_2(X_2^-)=X_2^-$, and $\sigma_2(X_1^+)=X_1^+$. Define a $\sigma_2$-derivation $\delta_2$ of $T_1$ by $\delta_2(r) = \delta_2(z) = \delta_2(X_1^-)=0$, $\delta_2(X_2^-) = \sigma_2(t_2)-t_2=-1$, and $\delta_2(X_1^+)=-z$. Set $T = T_1[X_2^+;\sigma_2,\delta_2]$.

As in Remark \ref{rmk.hberg}, the generator $z$ is superfluous. That is, using the relation $X_1^+X_2^+- X_2^+X_1^+=z$ to replace $z$ gives the necessary relations for $A$. First we check this for $X_1^+$ and $X_2^+$,
\begin{align*}
X_1^+z - z X_1^+
	&= X_1^+(X_1^+X_2^+-X_2^+X_1^+) - (X_1^+X_2^+- X_2^+X_1^+) X_1^+ \\
	&= (X_1^+)^2 X_2^+ - 2X_1^+ X_2^+ X_1^+ + X_2^+(X_1^+)^2, \\
X_2^+z -z X_2^+
	&= X_2^+(X_1^+X_2^+-X_2^+X_1^+) - (X_1^+X_2^+- X_2^+X_1^+) X_2^+ \\
	&= - \left( (X_2^+)^2 X_1^+- 2 X_2^+ X_1^+ X_2^+ + X_1^+(X_2^+)^2	\right).
\end{align*}
Now we check that this substitution does not in fact create any new relations. We have
\begin{align*}
zX_1^- - X_1^- z
	&= (X_1^+X_2^+- X_2^+X_1^+)X_1^- - X_1^- (X_1^+X_2^+- X_2^+X_1^+) \\
	&= (X_1^+X_1^-)X_2^- - X_2^+(X_1^+X_1^-) -  (X_1^-X_1^+)X_2^+ +  X_2^+(X_1^-X_1^+) \\
	&= (X_1^-X_1^+ + 1)X_2^- - X_2^+(X_1^-X_1^+ + 1) -  (X_1^-X_1^+)X_2^+ +  X_2^+(X_1^-X_1^+)
	= 0, \\
zX_2^- - X_2^- z
	&= (X_1^+X_2^+- X_2^+X_1^+)X_2^- - X_2^- (X_1^+X_2^+- X_2^+X_1^+) \\
	&= X_1^+(X_2^+X_2^-) - (X_2^+X_2^-)X_1^+ - X_1^+(X_2^-X_2^+) + (X_2^-X_2^+)X_1^+ \\
	&= X_1^+(X_2^-X_2^+-1) - (X_2^-X_2^+-1)X_1^+ - X_1^+(X_2^-X_2^+) + (X_2^-X_2^+)X_1^+ = 0.
\end{align*}
It follows that $A \iso T/(X_1^-X_1^+-t_1,X_2^-X_2^+-t_2)$.
\end{proof}

\begin{corollary}
Let $A$ be as in Theorem \ref{thm.ore2} with $\phi=\phi_1$ or $\phi_2$ as in Hypothesis \ref{hyp.A2auto}. Then $A^\phi$ is a noetherian TGWA.
\end{corollary}
\begin{proof}
This follows from Theorem \ref{thm.A2}, Theorem \ref{thm.ore2}, and \cite[Corollary 1.12]{Mo}.
\end{proof}

\section{Weight modules}\label{sec.weight}

The goal here is to describe how weight modules of a (T)GWA behave when restricted to the fixed ring. Some of these results generalize those from \cite{JW}, wherein the authors considered finite-dimensional simple modules for rank one GWAs and their invariants.
First we collect some definitions, mostly from \cite{DGO,HR2}.

\begin{definition}\label{def.weightmod}Let $A$ be an algebra over a ring $R$. A left $A$-module $M$ is called an $R$-\emph{weight module} if
\[ M = \bigoplus_{\bbm \in \MaxSpec(R)} M_\bbm,
\qquad M_\bbm = \{v \in M ~|~ \bbm v = 0\}.\]
Define the \emph{support} of $M$ on $R$ to be 
\[ \Supp_R(M) = \{ \bbm \in \MaxSpec(R) ~|~ M_\bbm \neq 0\}.\]
We denote by $(A,R)\wmod$ the category of left modules for $A$ that are $R$-weight modules.
\end{definition}

For the remainder of the section, the following hypotheses/notation are in effect.

\begin{hypothesis}\label{hyp.weight}
Let $R$ be a finitely-generated, commutative $\kk$-algebra and let $A=A_\mu(R,\sigma,t)$ be a regular, $\mu$-consistent TGWA of rank $n$ over $R$. Let $\Sigma=\grp{\sigma_1,\hdots,\sigma_n}\subset\Aut(R)$.
\end{hypothesis}

The abelian group $\Sigma$ acts on $\MaxSpec(R)$. Notice that if $M\in (A,R)\wmod$, then $X_i^{\pm}M_\bbm \subset M_{\sigma^{\pm 1}_i(\bbm)}$ for each $\bbm \in \Supp_R(M)$. 

\begin{definition}
For $\cO\in\MaxSpec(R)/\Sigma$ we denote by $(A,R)\wmodO$ the full subcategory of modules $M\in (A,R)\wmod$ such that $\Supp_R(M)\subset \cO$.
\end{definition}

\begin{lemma}\label{lem.orbits}
We have
\[(A,R)\wmod\simeq \coprod_{\cO\in\MaxSpec(R)/\Sigma}(A,R)\wmodO.\]
\end{lemma}
\begin{proof}
This is well-known, see for example \cite{MT}.
\end{proof}
Suppose that $\phi$ is an automorphism of $A$ satisfying Hypothesis \ref{hyp.auto}. Let $B\subset A$ be the subalgebra generated by $R^\phi$ and $(X_i^{\pm})^{m_i}$, then $B=A_{\nu}(R^\phi,\tau,s)$ as in Theorem \ref{thm.fix}. We have the subgroup of automorphisms $T=\grp{\tau_1,\ldots,\tau_n}=\grp{\sigma_1^{m_1},\hdots,\sigma_n^{m_n}}\subset \Sigma$. By  Remark \ref{rem.phisig}, $\sigma_i(R^\phi)\subset R^\phi$ for all $i \in [n]$. It follows that, by restricting the automorphisms to the fixed ring, we get maps
\[T\hookrightarrow \Sigma\to \Aut(R^\phi).\]
Clearly any $\Sigma$-orbit in $\MaxSpec(R)$ (or $\MaxSpec(R^\phi)$) partitions into $T$-orbits.

We have $R^\phi\hookrightarrow R$. Since $R$, and hence $R^\phi$, are finitely generated commutative $\kk$-algebras, then by the Nullstellensatz we have a surjective pullback map 
\[\pi:\MaxSpec(R)\onto \MaxSpec(R^\phi), \qquad \bbm\mapsto \bbm^\phi=\bbm\cap R^\phi.\]

For all $\bbm\in\MaxSpec(R)$, the inclusion $\bbm^\phi\subset \bbm$ induces a map
\[ R^\phi/\bbm^\phi \to R/\bbm\]
which is injective because $R^\phi/\bbm^\phi$ is a field.
Notice that, for all $g\in\Sigma$, since $\phi|_R$ commutes with $g$, we have $g(\bbm^\phi)=(g(\bbm))^\phi$. Hence, the isomorphisms $R/\bbm\simeq R/g(\bbm)$ and $R^\phi/\bbm^\phi\simeq R^\phi/g(\bbm^\phi)$, induced by $g$, fit into the following commutative diagram:
\begin{equation}\label{eq.comm-diag}
\begin{tikzcd}        
    R^\phi/\bbm^\phi \arrow[hook]{r} \arrow{d} & R/\bbm  \arrow{d} \\
    R^\phi/g(\bbm^\phi) \arrow[hook]{r} & R/g(\bbm)
\end{tikzcd}
\end{equation}

\begin{lemma}\label{lem.fiber-pi}
For all $\bbm\in\MaxSpec(R)$, we have $\pi^{-1}(\pi(\bbm))=\{\phi^k(\bbm)~|~0\leq k\leq\ell-1\}$.
\end{lemma}
\begin{proof}
Notice that $\pi(\phi(\bbm))=\phi(\bbm)\cap R^\phi=\bbm\cap R^\phi=\pi(\bbm)$, so $\pi^{-1}(\pi(\bbm))\supset\{\phi^k(\bbm)~|~0\leq k\leq\ell-1\}$. Since $\phi$ has finite order, the opposite inclusion also holds. This is an exercise in \cite[\S 5]{AM} but we will provide a proof for the benefit of the reader. Suppose $\bbm,\bbm'\in\MaxSpec(R)$ are such that $\bbm\cap R^\phi=\bbm'\cap R^\phi$. Then for all $x\in \bbm'$ we have that $\prod_{i=0}^{\ell-1}\phi^i(x)\in\bbm'\cap R^\phi=\bbm\cap R^\phi\subset \bbm$. It follows that there is a $j$ such that $\phi^j(x)\in\bbm$. Equivalently, $x\in\phi^{-j}(\bbm)$. Hence $\bbm'\subset \bigcup_{i=0}^{\ell-1}\phi^i(\bbm)$, which, by \cite[Prop. 1.11(i)]{AM} implies that $\bbm'\subset \phi^i(\bbm)$ for some $i$. By maximality, $\bbm'=\phi^i(\bbm)$ for some $i$.
\end{proof}

\begin{lemma}\label{lem.weightres}
Suppose that $M$ is an $R$-weight module. We consider the action of $R^\phi$ on $M$ by restriction.
\begin{enumerate}
    \item For all $\bbm\in \Supp_R(M)$, $\Res^R_{R^\phi}(M_\bbm)\subset M_{\bbm^\phi}$. In particular $M$ is also an $R^\phi$-weight module. 
    \item For each $\bbm^\phi\in\Supp_{R^\phi}(M)$, we have an equality of $R^\phi$-weight modules \[M_{\bbm^\phi}=\Res^R_{R^\phi}\left(\bigoplus_{k\in\ZZ}M_{\phi^k(\bbm)}\right).\]
\end{enumerate}
\end{lemma}
\begin{proof}
For (1), it is enough to notice that if $v\in M_\bbm$, then $rv=0$ for all $r\in\bbm^\phi$. Then (2) follows directly from (1) and Lemma \ref{lem.fiber-pi}.
\end{proof}

In light of Lemma \ref{lem.weightres}(1), we can now consider the restriction functor for TGWAs
\[\Res^A_B:(A,R)\wmod\to (B,R^\phi)\wmod.\]

Let $\cO=\Sigma\cdot \bbm\in\MaxSpec(R)/\Sigma$ be an orbit and let $\overline{\cO}=\pi(\cO)=\Sigma\cdot \bbm^\phi\in\MaxSpec(R^\phi)/\Sigma$. We then have $\overline{\cO}=\bigsqcup_{i}\overline{\cO}_i$ with $\overline{\cO}_i\in \MaxSpec(R^\phi)/T$. Then, for all $M\in (A,R)\wmodO$, we have, by Lemma \ref{lem.orbits}, $\Res^A_B(M)=\oplus_{i}M^{i}$, 
with $M^{i}\in (B,R^\phi)\wmod_{\overline{\cO}_i}$.

\subsection{Rank 1}
Let $n=1$. We consider a rank $1$ TGWA $A=A(R,\sigma,t)$, which is in fact just a GWA. Notice that in this case we have $A^\phi=B=A(R^\phi,\tau,s)$ for any isomorphism $\phi$ as in Hypothesis \ref{hyp.auto}. We drop all the indices and let $\phi(X^{\pm})=\alpha^{\pm 1}X^{\pm}$ with $m=\ord(\alpha)$. Then 
\[A^\phi=\bigoplus_{k\in\ZZ}R^\phi (Z_k)^m\subset \bigoplus_{k\in\ZZ}R Z_k=A,
\quad\text{ with }\quad Z_k=\begin{cases}(X^+)^k & \text{ if }k\geq 0 \\ (X^-)^{-k} & \text{ if }k<0.\end{cases}\]
We now recall some results about the classification of simple weight modules from \cite{DGO}.

\begin{definition}
We say that $\bbm \in \MaxSpec(R)$ is a \emph{break} of $A$ if $t\in\bbm$. We denote by $\beta_A\subset\MaxSpec(R)$ the set of all breaks.
\end{definition}

We define an order of $\cO \in \MaxSpec(R)/\Sigma$ by $\bbm<\sigma(\bbm)$. Let $\beta_{A,\cO}=\beta_A\cap\cO$ be the set of breaks in the orbit $\cO$ and suppose $\beta_{A,\cO} \neq \emptyset$. Then we set $\beta_{A,\cO}'=\beta_{A,\cO} \cup \{\infty\}$ if $\beta_{A,\cO}$ contains a maximal element and $\beta_{A,\cO}'=\beta_{A,\cO}$ otherwise. For $\bbn\in\beta_{A,\cO}'$, we let $\bbn^-$ be the maximal element of $\beta_{A,\cO}'$ such that $\bbn^-<\bbn$, or $\bbn^-=-\infty$ if $\bbn$ was a minimal element of $\beta_{A,\cO}'$. We extend the order in $\cO$ to $\cO\cup\{\pm\infty\}$ in the obvious way.

\begin{theorem}[{\cite[\S 5.2 \S5.4]{DGO}}]\label{thm.dgo}
Let $\cO\in\MaxSpec(R)/\Sigma$ be an infinite orbit. Let $\beta_{A,\cO}\subset\cO$ be the set of breaks in the orbit $\cO$.
\begin{enumerate}
\item Suppose $\beta_{A,\cO}=\emptyset$. Up to isomorphism, there is a unique simple module $M^{\cO}\in(A,R)\wmod_\cO$. Moreover, $\Supp_R(M)=\cO$, $M^{\cO}\iso \oplus_{\bbm\in\cO}(R/\bbm)v_\bbm$ as $R$-modules, and the $A$-action is given by $X^{\pm} v_\bbm=v_{\sigma^{\pm 1}(\bbm)}$.

\item Suppose $\beta_{A,\cO}\neq \emptyset$ and let $\beta_{A,\cO}'$ be as above. Up to isomorphism, the simple weight modules in $(A,R)\wmod_\cO$ are $\{M_{[\bbn]}^{\cO}~|~\bbn\in\beta_{A,\cO}'\}$, where
\[\Supp_R(M_{[\bbn]})=\{\bbm\in\cO~|~\bbn^-<\bbm\leq \bbn\}\] (the second inequality is actually strict if $\bbn=\infty$), $M_{[\bbn]}\iso \oplus_{\bbm\in\Supp_R(M_{[\bbn]})}(R/\bbm) v_\bbm$ as $R$-modules, and the $A$-action is given by 
\[X^{\pm} v_\bbm=\begin{cases}v_{\sigma^{\pm 1}(\bbm)} & \text{ if }\sigma^{\pm 1}(\bbm)\in\Supp_R(M_{[\bbn]}) \\ 0 & \text{ otherwise.}\end{cases}\]
\end{enumerate}
\end{theorem}

\begin{example}
Assume $\mathrm{char}\kk=0$. Let $R=\kk[h]$, $\sigma(h)=h-1$, then any orbit $\cO\in\MaxSpec(R)$ is infinite and the breaks correspond to the irreducible factors of $t$.
\end{example}

Notice that $t=\phi(t)$, hence for $\bbm\in\MaxSpec(R)$ we have $t\in\bbm$ if and only if $t\in\bbm^\phi$. Since $A^\phi=A(R^\phi,\tau,s)$ with $s=\prod_{j=0}^m \sigma^{-j}(t)$, the set of breaks of $A^\phi$ in $\MaxSpec(R^\phi)$ is 
\[\beta_{A^\phi}=\pi\left(\bigcup_{j=0}^{m-1}\sigma^{-j}(\beta_A)\right)=\{\sigma^{-j}(\bbm^\phi)=(\sigma^{-j}(\bbm))^\phi ~|~ 0\leq j\leq m-1,~\bbm\in\beta_A\}.\]

\begin{proposition}\label{prop.infiniteorb}
Let $A=A(R,\sigma,t)$ be a rank $1$ TGWA, $\phi\in Aut(A)$ as in Hypothesis 3.1, and $\cO\in\MaxSpec(R)/\Sigma$ an infinite orbit. Then, for all $\bbm\in \cO$, $\left(\grp{\phi}\cdot\bbm\right)\cap \cO)=\{\bbm\}$. Further, for $\bbm, \bbm'\in \cO$ with $\bbm\neq \bbm'$, we have $\bbm^\phi\neq (\bbm')^\phi$.
\end{proposition}
\begin{proof}
Let $\bbm\in\MaxSpec(R)$, with $\cO=\Sigma\cdot \bbm$ and suppose that for some $k\in \ZZ$, $\phi^k(\bbm)\in \cO$. Hence $\phi^k(\bbm)=\sigma^j(\bbm)$ for some $j\in\ZZ$. Recall that $\phi^\ell=\id_R$. Then we have
\[ \bbm = (\phi^{\ell})^k(\bbm) = (\phi^k)^{\ell}(\bbm) = (\sigma^j)^{\ell}(\bbm) = \sigma^{j\ell}(\bbm).\]
Since $\cO$ is infinite, we have $j\ell=0$, so $j=0$. This implies that $\phi^k(\bbm)=\sigma^j(\bbm)=\bbm$. Finally, if $\bbm,\bbm'\in \cO$ are such that $\bbm^\phi=(\bbm')^\phi$, by Lemma \ref{lem.fiber-pi} $\bbm'\in \grp{\phi}\cdot \bbm$, so $\bbm=\bbm'$.
\end{proof}

\begin{theorem}\label{thm.rank1}
Let $A=A(R,\sigma,t)$ be a rank $1$ TGWA, $\phi\in Aut(A)$ as in Hypothesis 3.1, and $\cO=\Sigma\cdot\bbm \in\MaxSpec(R)/\Sigma$ an infinite orbit. Let $d_\phi=[R/\bbm:R^\phi/\bbm^\phi]$ be the degree of the field extension.
\begin{enumerate}
\item \label{rank1-1} Suppose $\cO$ does not contain any breaks and let $M\in(A,R)\wmod_\cO$ be the unique simple weight module. Then 
\[\Res^A_{A^\phi}(M)=\bigoplus_{i=0}^{m-1} \left(M^{\overline{\cO}_i}\right)^{\oplus d_\phi}\]
where $\overline{\cO}_i=T\cdot\sigma^i(\bbm)=\{\sigma^{km+i}(\bbm)\}_{k\in\ZZ}$ is an infinite orbit without breaks in $\MaxSpec(R^\phi)/T$ and $M^{\overline{\cO}_i}\in (A^\phi,R^\phi)\wmod_{\overline{\cO}_i}$ is the unique simple weight module supported on that orbit.

\item \label{rank1-2} Suppose $\cO$ contains breaks. As in Theorem \ref{thm.dgo}(2), let $M_{[\bbn]}\in(A,R)\wmod_\cO$ be the unique simple module with $\Supp_R(M_{[\bbn]})=\{\bbm\in\cO~|~\bbn^-<\bbm\leq \bbn\}$. Then $\overline{\cO}_i=T\cdot\sigma^i(\bbm)$ is a $T$-orbit as above and we let $M_{[\bbn]}^{\overline{\cO}_i}\in(A^\phi,R^\phi)\wmod_{\overline{\cO}_i}$ be the unique simple weight module for $A^\phi$ with $\Supp_{R^\phi}(M_{[\bbn]}^{\overline{\cO}_i})=\pi(\Supp_R(M_{[\bbn]}))\cap\overline{\cO}_i$ (which, when it is not empty, is an interval between breaks for $s$). We then have
\[\Res^A_{A^\phi}(M_{[\bbn]})=\bigoplus_i\left(M_{[\bbn]}^{\overline{\cO}_i}\right)^{\oplus d_\phi}\]
where the direct sum is over those $i$ such that $\pi(\Supp_R(M_{[\bbn]}))\cap\overline{\cO}_i\neq\emptyset$.
\end{enumerate}
\end{theorem}
\begin{proof}
Since $\cO=\{\sigma^k(\bbm)~|~k\in\mathbb{Z}\}$ is an infinite $\Sigma$-orbit, it follows from Proposition \ref{prop.infiniteorb} that \[\overline{\cO}=\pi(\cO)=\{(\sigma^k(\bbm))^\phi=\sigma^k(\bbm^\phi)~|~k\in\ZZ\}\] 
is also an infinite orbit. Since $T=\grp{\sigma^m}$, it is clear that the decomposition of $\overline{\cO}$ into $T$-orbits is \[\overline{\cO}=\bigsqcup_{i=0}^{m-1}\overline{\cO}_i=\bigsqcup_{i=0}^{m-1}\{\sigma^{km+i}(\bbm)~|~k\in\mathbb{Z}\}\] 
with each $\overline{\cO}_i$ being an infinite orbit as well. We can choose elements $\{e_j\}_{j\in J}$ such that $e_j\in R$ and their images in the quotient $\bar{e}_j\in R/\bbm$ are a basis of $R/\bbm$ over $R^\phi/\bbm^\phi$. Then, by \eqref{eq.comm-diag} we have that, for all $k\in \mathbb{Z}$, $\{\sigma^k(\bar{e}_j)\}_{j\in J}=\{\overline{\sigma^k(e_j)}\}_{j\in J}$ is a basis of $R/\sigma^k(\bbm)$ over $R^\phi/\sigma^k(\bbm^\phi)$.

\eqref{rank1-1} 
If $\beta_{A,\cO}=\emptyset$, then $\beta_{A^\phi}\cap\overline{\cO}=\emptyset$. Hence, $\beta_{A^\phi}\cap\overline{\cO}_i=\emptyset$ for all $i=0,\ldots,m-1$ so the $\overline{\cO}_i$ are infinite orbits without breaks. For each $i=0,\ldots, m$ and for each $j\in J$, $A^\phi\cdot \sigma^i(\bar{e}_j)v_{\sigma^i(\bbm)}$ is a $(A^\phi,R^\phi)$-weight module supported on $\overline{\cO}_i$ and its component of weight $\sigma^i(\bbm)$ is $A^\phi_0\cdot \sigma^i(\bar{e}_j)v_{\sigma^i(\bbm)}\simeq (R^\phi/\sigma^i(\bbm^\phi))\sigma^i(\bar{e}_j)v_{\sigma^i(\bbm)}$, which is one dimensional over $R^\phi/\sigma^i(\bbm^\phi)$. It follows that the composition series of $A^\phi\cdot \sigma^i(\bar{e}_j)v_{\sigma^i(\bbm)}$ consists of a single copy of the unique simple module $M^{\overline{\cO}_i}$, hence $A^\phi\cdot \sigma^i(\bar{e}_j)v_{\sigma^i(\bbm)}\simeq M^{\overline{\cO}_i}$ for all $j\in J$. 

It suffices to prove that the sum of the $M^{\overline{\cO}_i}$ is direct. Let $j,j'\in J$, $j\neq j'$, and set $N:=A^\phi\cdot \sigma^i(\bar{e}_j)v_{\sigma^i(\bbm)}\cap A^\phi\cdot \sigma^i(\bar{e}_{j'})v_{\sigma^i(\bbm)}\neq 0$. We claim $N=0$. Suppose otherwise. Then $N$ is a weight module supported on $\overline{\cO}_i$, so there exists $0 \neq x \in (N)_{\sigma^{km+i}(\bbm)}$ for some $k\in\ZZ$. This implies that there are $r,r'\in R^\phi$ such that
\[
rZ_{km}\sigma^i(\bar{e}_j)v_{\sigma^i(\bbm)}=x=r'Z_{km}\sigma^i(\bar{e}_{j'})v_{\sigma^i(\bbm)}
\]
with $Z_{km}=(X^+)^{km}$ if $k\geq 0$ and $Z_{km}=(X^-)^{-km}$ if $k<0$. Then
\[
r\sigma^{km+i}(\bar{e}_j)v_{\sigma^{km+i}(\bbm)}
    = rZ_{km}\sigma^i(\bar{e}_j)v_{\sigma^i(\bbm)}
    = r'Z_{km}\sigma^i(\bar{e}_{j'})v_{\sigma^i(\bbm)}
    = r'\sigma^{km+i}(\bar{e}_{j'})v_{\sigma^{km+i}(\bbm)},
\]
but this is impossible because $\sigma^{km+i}(\bar{e}_j)$ and $\sigma^{km+i}(\bar{e}_{j'})$ are linearly independent over $R^\phi$. We have shown that $\Res^A_{A^\phi}(M)\supset \bigoplus_{i=0}^{m-1} \left(M^{\overline{\cO}_i}\right)^{\oplus d_\phi}$ and the equality follows because the weight spaces on each side have the same dimension over $R^\phi/\sigma^k(\bbm^\phi)$.

\eqref{rank1-2} Now suppose that $\beta_{A,\cO}\neq \emptyset$, and let $\sigma^k(\bbm)\in\beta_{A,\cO}$, then for all $j=0,\ldots,m-1$ we have $\pi(\sigma^{-j}(\sigma^k(\bbm))=\sigma^{k-j}(\bbm^\phi)\in \beta_{A^\phi}\cap\overline{\cO}$. In particular, since $k-j$ with $j=0,\ldots,m-1$ varies over all congruence classes modulo $m$, we have that $\beta_{A^\phi}\cap\overline{\cO}_i\neq\emptyset$ for all $i=0,\ldots,m$, so all those are infinite orbits with breaks. Then the same reasoning from part (1) applies here, in terms of the decomposition of the modules, once we take into account their supports. 
To conclude we need only verify that $\pi(\Supp_R(M_{[\bbn]}))\cap\overline{\cO}_i$ is indeed an interval between breaks for $s$. First of all, notice that because of Proposition \ref{prop.infiniteorb} and the fact that for all $k$, $\sigma^k(\bbm^\phi)=(\sigma^k(\bbm))^\phi$, the order on $\cO$ descends to the order naturally defined on $\overline{\cO}$ and is then restricted to the order on $\overline{\cO}_i$ for each $i=0,\ldots,m$. So for all $\bbn\in\beta_{A,\cO}'$ we have
\[\pi(\Supp_R(M_{[\bbn]}))=\{\bbm^\phi~|~(\bbn^-)^\phi<\bbm^\phi\leq \bbn^\phi\}\] 
(by convention here $(\pm\infty)^\phi=\pm\infty$). Now, when $\bbn^-,\bbn\neq\pm\infty$, there exist unique $j_1, j_2\in \ZZ$ such that $0\leq j_1,j_2\leq m-1$ and $\sigma^{-j_1}((\bbn^-)^\phi), \sigma^{-j_2}(\bbn^\phi)\in\overline{\cO}_i$. If $\pi(\Supp_R(M_{[\bbn]}))\cap\overline{\cO}_i\neq\emptyset$, then $\sigma^{-j_1}((\bbn^-)^\phi)<\sigma^{-j_2}(\bbn^\phi)$  and no other breaks for $s$ are in $\overline{\cO}_i$ between those two, so 
\[\pi(\Supp_R(M_{[\bbn]}))\cap\overline{\cO}_i=\{\bbm^\phi\in\overline{\cO}_i~|~\sigma^{-j_1}((\bbn^-)^\phi)<\bbm^\phi\leq \sigma^{-j_2}(\bbn^\phi)\}\]
is an interval as required. If either $\bbn^-=-\infty$ or $\bbn=\infty$, the same applies, except that $\sigma^{-j_1}$ or $\sigma^{-j_2}$ are not needed.
\end{proof}

\begin{example}
Let $R=\CC[h]$, $\sigma(h)=h-1$, $t=h(h -2)$, then we define the automorphism $\phi$ of $A(R,\sigma,t)$ by $\phi(h)=h$, $\phi(X^\pm)=\alpha^{\pm 1}X^\pm$ with $\alpha=e^{\frac{2\pi i}{3}}$, so $m=3$. Let $\bbm=(h)$, $\cO=\Sigma\cdot\bbm=\{(h-k)~|~k\in\ZZ\}$ which is an infinite orbit with breaks $\{(h),(h-2)\}$. Then, since $R^\phi=R$, 
$$\overline{\cO}=\cO=\bigsqcup_{i=0}^2 T\cdot \sigma^i\bbm=\bigsqcup_{i=0}^2 \{(h-k)~|~k\equiv i\mod 3\}=\bigsqcup_{i=0}^2\overline{\cO}_i.$$
In the diagram below we have a part of the orbit $\cO$, which is subdivided into $T$-orbits, where $\overline{\cO}_0$ is denoted by circles, $\overline{\cO}_1$ is denoted by triangles, and $\overline{\cO}_2$ is denoted by squares. We have also made the shape hollow if the corresponding ideal is a break for $t$. The breaks for $s=\sigma^{-2}(t)\sigma^{-1}(t)t$ are then $\{(h+2), (h+1), (h), (h-1), (h-2)\}$.

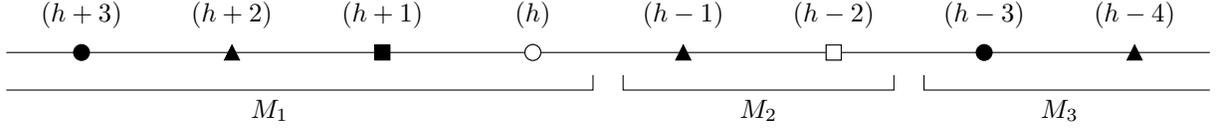
\vspace{.5cm}

\begin{center}
\begin{tikzpicture}
\draw[-] (-7,0) -- (9,0) ;
\foreach \x in {3,2,1}{
\draw(-2*\x,0.2) node[above] {$(h+\x)$};}
\foreach \x in {4,3,2,1}{
\draw(2*\x,0.2) node[above] {$(h-\x)$};}
\draw(0,0.2) node[above] {$(h)$};
\foreach \x in {-6,0,6}{
\filldraw(\x,0) circle(3pt); }
\filldraw[white] (0,0) circle(3pt);
\draw (0,0) circle(3pt);
\foreach \x in {-4,2,8}{
\filldraw (\x,0.1) -- (\x+0.1,-0.1) -- (\x-0.1,-0.1) -- cycle; }
\foreach \x in {-2,4}{
\filldraw (\x+0.1,0.1) -- (\x+0.1,-0.1) -- (\x-0.1,-0.1) -- (\x-0.1,0.1) -- cycle; }
\filldraw[white] (4+0.1,0.1) -- (4+0.1,-0.1) -- (4-0.1,-0.1) -- (4-0.1,0.1) -- cycle;
\draw (4+0.1,0.1) -- (4+0.1,-0.1) -- (4-0.1,-0.1) -- (4-0.1,0.1) -- cycle;
\draw (-7,-0.5) -- (0.8,-0.5) -- (0.8,-0.3);
\draw (-3.5,-0.5) node[below] {$M_1$};
\draw (1.2,-0.3) -- (1.2,-0.5) -- (4.8,-0.5) -- (4.8,-0.3);
\draw (3,-0.5) node[below] {$M_2$};
\draw (9,-0.5) -- (5.2,-0.5) -- (5.2,-0.3);
\draw (7,-0.5) node[below] {$M_3$};
\end{tikzpicture}
\end{center}

\vspace{.2cm}

Then, there are 3 simple weight modules for $A$, $M_p$, $p=1,2,3$ with $\Supp_R(M_1)=\{(h-k)~|~k\in\ZZ_{\leq 0}\}$, $\Supp_R(M_2)=\{(h-1),(h-2)\}$, $\Supp_R(M_3)=\{(h-k)~|~k\geq 3\}$. The action of $X^+$ shifts weight spaces to the right, while the action of $X^-$ shifts to the left. The action by $A^\phi$ is then given by $(X^{\pm})^3$ so it can be thought of as a shift by three places, until the next space with the same shape. When we restrict to $A^\phi$ we have then $\Res^A_{A^\phi}(M_p)=M_p^{\overline{\cO}_0}\oplus M_p^{\overline{\cO}_1}\oplus M_p^{\overline{\cO}_2}$ with $\Supp_R(M_p^{\overline{\cO}_i})=\Supp(M_p)\cap\overline{\cO}_i.$ Note that $M_2^{\overline{\cO}_0}=0$ because the support of $M_2$ does not intersect that $T$-orbit. Also note that, for example, $(h+1)$ is a break for $s$, so $\Supp_R(M_1^{\overline{\cO}_2})=\{\bbm\in \overline{\cO}_2~|~-\infty<\bbm\leq (h+1)\}$ is indeed an interval between breaks for $s$. 
\end{example}

When the orbits are finite, things are more complicated (and potentially more interesting) because there can be non-isomorphic simple modules with the same support.

\begin{example}
Let $R=\CC[h]$, $\sigma(h)=ih$, $t=h^2-1$, then the rank $1$ GWA $A=A(R,\sigma,t)$ has an automorphism $\phi$ defined by $\phi(h)=-h$, $\phi(X^\pm)=X^\pm$. In this case, $A^\phi=A(R^\phi,\tau,s)$, with $R^\phi=\CC[h^2]$, $\tau=\sigma$, $s=t$. We take $\bbm=(h-2)$, then 
\[\cO=\Sigma\cdot\bbm=\{\sigma^k(\bbm)\}_{k=0}^3=\{(h-2), (h-2i), (h+2), (h+2i)\}\]
is a finite orbit, and we have, for $c=2,2i,-2,-2i$, that $(\grp{\phi}\cdot(h-c))\cap \cO=\{(h-c),(h+c)\}$, hence $\left(\CC[h](h-c)\right)^\phi=\CC[h^2](h^2-c^2)=\left(\CC[h](h+c)\right)^\phi$, illustrating why Proposition \ref{prop.infiniteorb} does not apply to finite orbits.
Moreover, for each $z\in\CC^\times$ we can define a simple weight module $M\in(A,R)\wmod_{\cO}$ for $A$ with basis $\{v_2,v_{2i},v_{-2},v_{-2i}\}$, where $(h-c)v_c=0$ and, with the basis in that order, we have the linear operators
\[X^-=\begin{bmatrix} 0 & 0 & 0 & z \\
1 & 0 & 0 & 0 \\
0 & 1 & 0 & 0 \\
0 & 0 & 1 & 0 
\end{bmatrix}\qquad 
X^+ =\begin{bmatrix} 0 & -5 & 0 & 0 \\
0 & 0 & 3 & 0 \\
0 & 0 & 0 & -5 \\
3/z & 0 & 0 & 0 
\end{bmatrix}.
\]
Then 
\[\pi(\cO)=\overline{\cO}=\Sigma\cdot\bbm^\phi=\{(h^2-4), (h^2+4)\}\]
and $\Res^A_{A^\phi}(M)=M_+\oplus M_-$. Here $M_\pm \in (A^\phi,R^\phi)\wmod_{\overline{\cO}}$ is a simple weight module of dimension $2$ with basis $\{v_{-2}\pm\sqrt{z}v_2,v_{-2i}\pm\sqrt{z}v_{2i}\}$, (notice that those are not weight vectors for $R$, but they are weight vectors for $R^\phi$), and the action of $X^+$ and $X^-$ on $M_{\pm}$ is given by
\[ X^-=\begin{bmatrix} 0 & \pm\sqrt{z} \\
1 & 0
\end{bmatrix} \qquad X^+=\begin{bmatrix} 0 & -5 \\ \pm3/\sqrt{z} & 0 \end{bmatrix}.\]
This shows that the situation here is different from Theorem \ref{thm.rank1}, since considering the restriction of a simple module and taking the component with support in a given orbit of $\MaxSpec(R^\phi)$ does not yield several copies of the same simple module but rather the sum of two non isomorphic simples.
\end{example}

\subsection{Rank 2}
Let $n=2$, $R=\kk[h]$, $\sigma_i(h)=h-\beta_i$, $i=1,2$, $t=(p_1,p_2)$ as in Example \ref{ex.rank2}, then we have the noncommutative fiber product TGWA $A=A(R,\sigma,t)$. Simple weight modules for such TGWAs have been classified in \cite{hart4}. As a quick summary of those results, for each orbit $\cO\in\MaxSpec(R)/\Sigma$, $t$ corresponds to a union of lattice paths (which topologically are loops) on a cylinder and the simple weight modules in $(A,R)\wmod_\cO$ are $M_{\cD,\xi}$, where $\cD$ is a connected component of the cylinder minus the loops, and $\xi\in\kk$ with $\xi= 0$ if and only if $\cD$ is contractible. Here $\cD$ corresponds to the support of the weight module, while $\xi$ is the scalar giving the action of a certain centralizing element $\cC$. 

When $\beta_1=1=-\beta_2$, and $p_1=h$, $p_2=h-1$ then $A$ is a $\kk$-finitistic TGWA of Type $A_2$. We can then take $\phi=\phi_1$ or $\phi_2$ as in Hypothesis \ref{hyp.A2auto} and the resulting fixed ring $A^{\phi}$ is again a noncommutative fiber product TGWA by Theorem \ref{thm.A2} (the fixed ring is no longer of type $A_2$ though). Notice that, since $\phi$ has to commute with $\sigma_1$ and $\sigma_2$, we have necessarily that $\phi|_R=\id_R$ ($\phi$ has to be a shift to commute with $\sigma_i$, and it has to leave $p_1$ and $p_2$ invariant). We only consider the orbit $\cO=\{(h-k)~|~k\in\ZZ\}$ (all other orbits do not have any breaks so they have a single simple weight module with full support). In this case the cylinder, drawn as a plane strip where the points $(1,y)$ are identified with the points $(0,y-1)$, looks as follows.

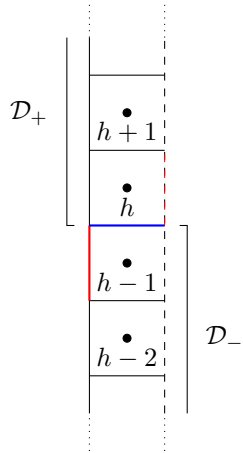
\begin{center}
\begin{tikzpicture}[scale=1]
\draw[-] (0,-0.5) -- (0,4.5);
\draw[dotted] (0,-1) -- (0,-0.5);
\draw[dotted] (0,4.5) -- (0,5);
\draw[dashed] (1,-0.5) -- (1,4.5);
\draw[dotted] (1,-1) -- (1,-0.5);
\draw[dotted] (1,4.5) -- (1,5);
\foreach \x in {0,1,...,4}
{\draw[-] (0,\x) -- (1,\x);
}
\foreach \x in {1,2}
{\filldraw (0.5,2.5-\x) circle(1.5pt) node[below]{$h-\x$};
}
\filldraw (0.5,2.5) circle(1.5pt) node[below]{$h$};
\filldraw (0.5,3.5) circle(1.5pt) node[below]{$h+1$};
\draw[thick,red] (0,1) -- (0,2);
\draw[thick,blue] (1,2) -- (0,2);
\draw[dashed,red] (1,2) -- (1,3);
\draw[-] (-0.2,2) -- (-0.3,2) -- (-0.3,4.5);
\draw[-] (1.2,2) -- (1.3,2) -- (1.3,-0.5);
\draw (-0.8,3.5) node {$\cD_+$};
\draw (1.8,0.5) node {$\cD_-$};
\end{tikzpicture}
\end{center}
Here the dots correspond to the weight spaces, $\sigma_1$ corresponds to moving one spot to the right (down by the identification of points on the cylinder), $\sigma_2$ corresponds to moving up one spot (or to the left). We draw edges to make lattice paths from $t=(p_1,p_2)$ by selecting the vertical edges to the right of the weight spaces of the factors of $p_1$ and the horizontal edges above the weight spaces of the factors of $p_2$. Hence, the red edge corresponds to $p_1=h$ (we actually solidly drew its identification on the left boundary of the vertical strip and dotted the one on the right boundary), the blue edge corresponds to $p_2=h-1$ and together they form a lattice path that cuts the cylinder into two non-contractible connected components $\cD_+$ and $\cD_-$. 
By \cite[Prop. 6.3]{hart4} the centralizing element $\cC$ in this case actually belongs to $A$ and not only to its localization by $\kk(h)$, and can be written in equivalent ways as
\begin{equation}\label{eq.defc} \cC=X_1^+X_2^+\frac{1}{h-1}=X_2^+X_1^+\frac{1}{h}=X_2^+X_1^+-X_1^+X_2^+.\end{equation}

Now, consider $\phi=\phi_2$ as in Hypothesis \ref{hyp.A2auto}, the case of $\phi_1$ being completely analogous. Since $\alpha_2=1$, we let $\alpha=\alpha_1$, with multiplicative order $m=m_1$, then $A^{\phi}$ is the rank 2 noncommutative fiber product TGWA $A(R^\phi,\tau,s)$ with $R^\phi=R=\kk[h]$, $\tau_1(h)=\sigma_1^m(h)=h-m$, $\tau_2(h)=\sigma_2(h)=h+1$, $s_1=\prod_{j=0}^{m-1}\sigma_1^{-j}(p_1)=\prod_{j=0}^{m-1}(h+j)$, $s_2=p_2=h-1$. For $i=1,2$, we denote by $Y_i^{\pm}$ the generators of $A^{\phi}$, so that $Y_1^\pm=(X_1^\pm)^m$, and $Y_2^\pm=X_2^\pm$. Notice that $T=\Sigma$, hence $\overline{\cO}=\cO$.

Then, the cylinder for $A^{\phi}$ looks as below, where the points $(1,y)$ are identified with the points $(0,y-m)$.

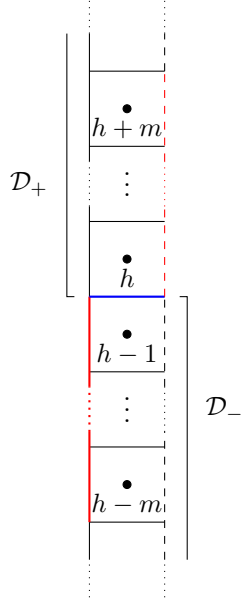
\begin{center}
\begin{tikzpicture}[scale=1]
\draw[-] (0,-0.5) -- (0,0);
\draw[-] (0,3) -- (0,4.2);
\draw[-] (0,4.8) -- (0,6.5);
\draw[dotted] (0,4.2) -- (0,4.8);
\draw[dotted] (0,-1) -- (0,-0.5);
\draw[dotted] (0,6.5) -- (0,7);
\draw[dashed] (1,-0.5) -- (1,1.2);
\draw[dashed] (1,1.8) -- (1,3);
\draw[dotted] (1,1.2) -- (1,1.8);
\draw[dashed] (1,6) -- (1,6.5);
\draw[dotted] (1,-1) -- (1,-0.5);
\draw[dotted] (1,6.5) -- (1,7);
\foreach \x in {0,1,...,6}
{\draw[-] (0,\x) -- (1,\x);
}
\filldraw (0.5,0.5) circle(1.5pt) node[below]{$h-m$};
\draw (0.5,1.6) node{$\vdots$};
\draw (0.5,4.6) node{$\vdots$};
\filldraw (0.5,3.5) circle(1.5pt) node[below]{$h$};
\filldraw (0.5,2.5) circle(1.5pt) node[below]{$h-1$};
\filldraw (0.5,5.5) circle(1.5pt) node[below]{$h+m$};
\draw[thick,red] (0,1.8) -- (0,3);
\draw[thick,dotted, red] (0,1.2) -- (0,1.8);
\draw[thick,red] (0,0) -- (0,1.2);
\draw[thick,blue] (1,3) -- (0,3);
\draw[dashed,red] (1,3) -- (1,4.2);
\draw[dashed,red] (1,4.8) -- (1,6);
\draw[dotted,red] (1,4.2) -- (1,4.8);
\draw[-] (-0.2,3) -- (-0.3,3) -- (-0.3,6.5);
\draw[-] (1.2,3) -- (1.3,3) -- (1.3,-0.5);
\draw (-0.8,4.5) node {$\cD_+$};
\draw (1.8,1.5) node {$\cD_-$};
\end{tikzpicture}
\end{center}
Here $\tau_1$ corresponds to moving one spot to the right (or $m$ spots down, by the identification of points on the cylinder), $\tau_2=\sigma_2$ corresponds to moving up one spot. The red edges are (to the right of) the factors of $s_1$, the blue edge is (above) the one factor of $s_2=p_2$ and together they form a lattice path that cuts the cylinder into two non contractible connected components. We use the same notation of $\cD_+$ and $\cD_-$ as above because they consist of the same weight spaces. By \cite[Prop. 6.3]{hart4} the centralizing element $\cC^\phi$ actually belongs to $A^{\phi}$, and can be written (among other equivalent ways) as
\[\cC^\phi=Y_1^+(Y_2^+)^m\frac{1}{ (h-m)(h-m+1)\cdots(h-1)}=(Y_2^+)^mY_1^+\frac{1}{h(h+1)\cdots(h+m-1)}.\]

\begin{proposition}\label{prop.rank2mod}
Let $\phi=\phi_2$ as in Hypothesis \ref{hyp.A2auto}.
If we denote the simple weight modules of $A^{\phi}$ by $M^{\phi}_{\cD,\eta}$, where $\cD$ is the support identified as a connected component of the cylinder, and $\eta$ is given by the action of $\cC^\phi$, then we have \[\Res^A_{A^{\phi}}M_{\cD,\xi}=M^\phi_{\cD,\xi^m}.\]
\end{proposition}
\begin{proof}
The fact that the support is the same when restricting is clear. The only thing that needs to be checked is the action of $\cC^\phi$. We claim that $\cC^\phi=\cC^m$ in $A$. We will actually do the computation in the localization $A\otimes_{\kk[h]}\kk(h)$, but since both $\cC^\phi$ and $\cC^m$ belong to $A$, the equality stands there as well.

From \eqref{eq.tgwc1}-\eqref{eq.tgwc2}, we obtain the relation $X_1^+X_2^+h=X_2^+X_1^+(h-1)$ which can be written in the localization as 
\begin{equation}\label{eq.exchange}
X_1^+X_2^+=X_2^+X_1^+\frac{h-1}{h}.
\end{equation}
By induction, we claim that for $k>0$
\begin{equation}\label{eq.exchange-gen}
X_1^+(X_2^+)^k=(X_2^+)^kX_1^+\frac{h-k}{h}.
\end{equation}
It is clear for $k=1$, then
\begin{align*}
X_1^+(X_2^+)^{k+1}&=X_1^+(X_2^+)^kX_2^+ \\
&\stackrel{\mathclap{\text{(ind hyp)}}}{=} \quad (X_2^+)^kX_1\frac{h-k}{h}X_2^+ \\
&= (X_2^+)^kX_1X_2^+\frac{h-k-1}{h-1}\\
&\stackrel{\mathclap{\eqref{eq.exchange}}}{=} \quad
(X_2^+)^kX_2^+X_1\frac{h-1}{h}\cdot\frac{h-k-1}{h-1}\\
&=(X_2^+)^kX_1^+\frac{h-(k+1)}{h}.
\end{align*}
Now we claim that, for all $k>0$
\begin{equation}\label{eq.ctok}
\cC^k=(X_2^+)^k(X_1^+)^k\frac{1}{h(h+1)\cdots(h+k-1)}.    
\end{equation}
Again we proceed by induction, for $k=1$ this is clear by \eqref{eq.defc}. Then
\begin{align*}
\cC^{k+1}&=\cC\cdot \cC^k \\
&\stackrel{\mathclap{\text{(ind hyp)}}}{=} \quad
X_2^+X_1^+\frac{1}{h}(X_2^+)^k(X_1^+)^k\frac{1}{h(h+1)\cdots(h+k-1)} \\
&= X_2^+X_1^+(X_2^+)^k(X_1^+)^k\frac{1}{h^2(h+1)\cdots(h+k-1)}\\
&\stackrel{\mathclap{\eqref{eq.exchange-gen}}}{=} \quad
X_2^+(X_2^+)^kX_1^+\frac{h-k}{h}(X_1^+)^k\frac{1}{h^2(h+1)\cdots(h+k-1)}\\
&= (X_2^+)^{k+1}(X_1^+)^{k+1}\frac{h}{h+k}\cdot \frac{1}{h^2(h+1)\cdots(h+k-1)}\\
&=(X_2^+)^{k+1}(X_1^+)^{k+1}\frac{1}{h(h+1)\cdots(h+k-1)(h+k)}.
\end{align*}
By taking $k=m$ in \eqref{eq.ctok}, we then get that 
$$\cC^\phi=(Y_2^+)^mY_1^+\frac{1}{h(h+1)\cdots(h+m-1)}=(X_2^+)^{m}(X_1^+)^{m}\frac{1}{h(h+1)\cdots(h+m-1)}=\cC^m$$ 
and the result follows because if $\cC$ acts on a module by $\xi$, in the restriction $\cC^\phi$ will act by $\xi^m$.
\end{proof}

\bibliographystyle{plain}

\end{document}